\documentclass[10pt,a4paper]{article}

\newif\ifsiamart
\makeatletter%
\@ifclassloaded{siamart171218}%
  {\siamarttrue}%
  {\siamartfalse}%
\makeatother%

\ifsiamart\else
\usepackage{authblk}
\fi
\usepackage{silence}
\usepackage[T1]{fontenc}
\usepackage[utf8]{inputenc}
\usepackage{csquotes}
\usepackage[UKenglish]{babel}
\usepackage{paralist}
\usepackage[shortlabels]{enumitem}
\usepackage[margin=.6in]{geometry}
\usepackage[dvipsnames]{xcolor}
\usepackage{graphicx}
\usepackage{bbm}
\usepackage{amsmath}
\usepackage{amssymb}
\usepackage{amsfonts}
\usepackage{stmaryrd}
\usepackage{mathrsfs}
\usepackage{mathtools}
\usepackage{epstopdf}
\usepackage{comment}

\usepackage{xspace}
\usepackage{float}
\usepackage{stackengine}

\ifsiamart\else
\usepackage{amsthm}
\usepackage{hyperref}
\hypersetup{%
    linktocpage=true,
    colorlinks=true,
    citecolor=red,
    linkcolor=DarkOrchid,
    urlcolor=blue,
}
\usepackage[nameinlink,capitalise]{cleveref}
\newcommand{\email}[1]{\href{mailto:#1}{#1}}

\fi

\usepackage{setspace}

\DeclarePairedDelimiter\abs{\lvert}{\rvert}
\newcommand{\norm}[1]{\left\lVert#1\right\rVert}

\DeclareMathOperator{\I}{I}

\DeclareMathOperator{\Law}{Law}
\DeclareMathOperator{\e}{e}

\DeclareMathOperator*{\trace}{tr}

\newcommand{\hessian}{\operatorname{D}^2}
\newcommand{\placeholder}{\mathord{\color{black!33}\bullet}}%
\newcommand{\range}[2]{\llbracket #1, #2 \rrbracket}

\renewcommand{\d}{\mathrm d}

\newcommand{\N}{\mathbf{N}}
\newcommand{\R}{\mathbf R}
\newcommand{\real}{\R}
\newcommand{\nat}{\N}
\newcommand{\proba}{\mathbf P}
\newcommand{\expect}{\mathbf{E}}
\renewcommand{\t}{\mathsf T}

\renewcommand{\leq}{\leqslant}
\renewcommand{\geq}{\geqslant}
\renewcommand{\le}{\leqslant}
\renewcommand{\ge}{\geqslant}

\ifsiamart
\newtheorem{remark}[theorem]{Remark}
\newtheorem{assumption}{Assumption}

\else
\theoremstyle{plain}

\newtheorem{theorem}{Theorem}
\newtheorem{lemma}{Lemma}
\newtheorem{corollary}[theorem]{Corollary}
\newtheorem{proposition}[theorem]{Proposition}

\newtheorem{remark}{Remark}

\newtheorem{assumption}{Assumption}

\fi

\crefname{lemma}{Lemma}{Lemmas}
\crefname{remark}{Remark}{Remarks}
\crefname{assumption}{Assumption}{Assumptions}
\crefname{proposition}{Proposition}{Propositions}
\crefname{section}{Section}{Sections}
\crefname{subsection}{Subsection}{Subsections}
\crefname{equation}{}{}
\Crefname{equation}{Equation}{Equations}
\newlist{lemmaenum}{enumerate}{3}
\setlist[lemmaenum]{label=(\alph*),ref=\,(\alph*)}
\crefname{lemmaenum}{Lemma}{Lemmas}
\newlist{assumpenum}{enumerate}{5}
\setlist[assumpenum]{label=(\alph*), font={\bfseries}}
\newlist{auxenum}{enumerate}{2}
\setlist[auxenum]{label=(\alph*),ref=(\alph*)}
\crefname{auxenumi}{Item}{Items}
\crefname{enumi}{}{}
\crefname{equation}{}{}
\crefname{assumpenumi}{}{}
\crefname{assumpenumii}{}{}
\Crefname{assumpenumi}{Assumption}{Assumptions}
\Crefname{assumpenumii}{Assumption}{Assumptions}
\Crefname{assumpenumii}{Assumption}{Assumptions}
\crefrangeformat{assumpenumi}{#3#1#4--#5#2#6}
\Crefname{lemmaenumi}{Part}{Parts}
\Crefname{figure}{Figure}{Figures}
\numberwithin{equation}{section}

\usepackage{tikz}
\usetikzlibrary{shapes.misc}
\usetikzlibrary{positioning}

\usepackage[url=true,backend=biber,style=trad-abbrv,url=false,doi=false,isbn=false]{biblatex}
\DeclareFieldFormat{volume}{volume \textbf{#1}}
\DeclareFieldFormat{pages}{}
\DeclareFieldFormat[article]{volume}{\textbf{#1}}

\ifsiamart\else
\let\oldparagraph=\paragraph
\renewcommand\paragraph[1]{\oldparagraph{#1.}}
\fi

\title{Sharp Propagation of Chaos for the Ensemble Langevin Sampler}
\author{U. Vaes\thanks{MATHERIALS team, Inria Paris \& CERMICS, \'Ecole des Ponts, France (\email{urbain.vaes@inria.fr})}}

\newcommand{\X}{\mathscr X}
\newcommand{\xn}[1]{X^{#1}}
\newcommand{\wn}[1]{W^{#1}}
\newcommand{\yn}[1]{Y^{#1}}
\newcommand{\xnl}[1]{\overline X^{#1}}
\newcommand{\mfl}{\overline{\rho}}
\newcommand{\emp}{\mu^J}
\newcommand{\empmfl}{\overline \mu^J}
\newcommand{\xl}{\overline X}
\renewcommand{\r}{r}
\newcommand{\cont}{\mathrm C}
\renewcommand{\l}{c_{\ell}}
\renewcommand{\u}{c_{u}}
\newcommand{\lipphi}{\widetilde \u}
\newcommand{\revise}[1]{\textcolor{black}{#1}}

\DeclareMathOperator{\mean}{\mathcal M}
\DeclareMathOperator{\cov}{\mathcal C}

\begin{document}
\maketitle


\begin{abstract}
    The aim of this paper is to revisit propagation of chaos for a Langevin-type interacting particle system recently proposed as a method to sample probability measures.
    The interacting particle system we consider coincides,
    in the setting of a log-quadratic target distribution,
    with the ensemble Kalman sampler~\cite{MR4059375},
    for which propagation of chaos was first proved by Ding and Li in~\cite{MR4234152}.
    Like these authors,
    we prove propagation of chaos with an approach based on a synchronous coupling,
    as in Sznitman's classical argument~\cite{MR1108185}.
    Instead of relying on a boostrapping argument,
    however, we use a technique based on stopping times in order to handle the presence of the empirical covariance in the coefficients of the dynamics.
    The use of stopping times to handle the lack of global Lipschitz continuity in the coefficients of stochastic dynamics originates from numerical analysis~\cite{MR1949404}
    and was recently employed to prove mean field limits for consensus-based optimization and related interacting particle systems~\cite{MR4553241,gerber2023meanfield}.
    In the context of ensemble Langevin sampling,
    this technique enables proving pathwise propagation of chaos with optimal rate,
    whereas previous results were optimal only up to a positive $\varepsilon$.
\end{abstract}

\section{Introduction}
\label{sec:intro}

In a wide variety of applications,
ranging from Bayesian inference to statistical physics and computational biology,
it is necessary to produce samples from high-dimensional probability distributions of the form
\begin{equation}
    \label{eq:target_measure}
    \mu = \frac{\e^{-\phi}}{Z}, \qquad Z = \int_{\real^d} \e^{-\phi}.
\end{equation}
where $\phi\colon \real^d \rightarrow \real$ is a given function such that $\e^{-\phi}$ is integrable.
In~\cite{MR4059375},
the authors propose to simulate the following interacting particle system in order to generate approximate samples from~$\mu$:
\begin{equation}
    \label{eq:eks}
    \d \xn{j}_{\revise{t}} = - \cov(\emp_t) \nabla \phi(\xn{j}_{\revise{t}}) \, \d t + \sqrt{2 \cov(\emp_t)} \, \d \wn{j}_t,
    \qquad j \in \range{1}{J},
\end{equation}
where
$(\wn{j})_{j \in \range{1}{J}}$ are independent standard Brownian motions in $\real^d$,
$\emp_t = \frac{1}{J} \sum_{j=1}^{J} \delta_{\xn{j}_t}$ is the associated empirical measure,
and~$\cov(\emp_t)$ denotes the covariance under $\emp_t$;
see~\eqref{eq:cov_mean_def} below for the precise definition.
\revise{%
    If the covariance~$\cov(\emp_t)$ in~\eqref{eq:eks} was replaced by a constant symmetric positive definite matrix,
    then~\eqref{eq:eks} would describe a system of independent overdamped Langevin diffusions
    in an external potential~$\phi$,
    and the system would be suitable for sampling from~$\mu$.
    The idea of the method~\eqref{eq:eks} is to include the covariance as a form of preconditioner,
    in order to accelerate convergence to the invariant distribution.
    The authors of~\cite{MR4059375}
}
also present a gradient-free approximation of~\eqref{eq:eks}
that is well-suited to those Bayesian inverse problems
for which it is difficult or undesirable to calculate derivatives of the forward model,
examples of which are given in~\cite{MR4442433}.
Taking formally the limit $J \to \infty$ in~\eqref{eq:eks},
they conjecture that the mean field limit of the system is given by the following McKean stochastic differential equation
\begin{equation}
    \label{eq:mckean}
    \left\{
    \begin{aligned}
        \d\xl_t &= - \cov (\mfl_t) \nabla \phi(\xl_t) \, \d t + \sqrt{2 \cov (\mfl_t)}  \, \d W_t,\\
        \mfl_t &= \Law(\xl_t)\,.
    \end{aligned}
    \right.
\end{equation}
For background material on mean field limits and propagation of chaos,
we refer to~\cite{ReviewChaintronI,ReviewChaintronII}.
The mean field equation~\eqref{eq:mckean} is often simpler to analyze mathematically than the interacting particle system~\eqref{eq:eks}.
\revise{%
    It is simple to show, notably, that the target probability distribution~$\mu$ is an invariant measure of~\eqref{eq:mckean}.
}
In the setting where $\phi$ is quadratic,
it is \revise{%
    also possible to prove that the law $\mfl_t$ converges exponentially as~$t \to \infty$
    to the target distribution~$\mu$ in the Euclidean Wasserstein metric,
    provided that the covariance under $\mfl_0$ is nondegenerate.
    Furthermore, in this case the convergence rate is independent of the mean and covariance parameters of the target measure~$\mu$~\cite{MR4246458,MR4059375},
    a property related to the affine invariance of the dynamics~\cite{MR4123680}.
}%
Although the interacting particle system~\eqref{eq:eks} is always simulated with a finite number of particles in practice,
mathematical study of the mean field equation~\eqref{eq:mckean} \revise{enables to gain insight} into the behavior of the particle system for large $J$.
\revise{Indeed}, once a quantitative version of propagation of chaos has been established,
error estimates in the finite-ensemble setting can be derived from those at the mean field level via \revise{triangle inequalities}.
\revise{This motivates the study of propagation of chaos for~\eqref{eq:eks}}.

\revise{Before focusing on~\eqref{eq:eks}, let us mention that}, in~\cite{nusken2019note},
a correction of~\eqref{eq:eks} is proposed to ensure that the interacting particle system possesses as invariant distribution the tensorized measure $\mu^{\otimes J}$
for every value of $J \geq d+1$.
The properties of the resulting interacting particle system,
with acronym~ALDI,
are analyzed in~\cite{MR4123680}.

\paragraph{Summary of previous work}
\revise{
    The presence of the empirical covariance,
    a quadratic quantity,
    is an obstruction for proving propagation of chaos rigorously for~\eqref{eq:eks},
    because McKean's and Sznitman's classical arguments assume Lipschitz interactions.
    The first proof of propagation of chaos for~\eqref{eq:eks} was proposed in~\cite{MR4234152} by Ding and Li,
    who investigated the particular situation where~$\phi$ is quadratic.
}
In the setting of Bayesian inverse problems~\cite{MR2652785},
this situation arises when the forward model is linear,
observational noise is Gaussian and the prior distribution is Gaussian.
In order to show convergence,
in an appropriate sense,
of the interacting particle system~\eqref{eq:eks} to the mean field limit~\eqref{eq:mckean},
Ding and Li employ as a pivot the following
synchronously coupled system,
with the same initial condition and the same Brownian motions:
\begin{equation}
    \label{eq:synchronous_coupling}
    \left\{
    \begin{aligned}
        \d \xnl{j}_{\revise{t}} &= - \cov(\mfl_t) \nabla \phi(\xnl{j}_{\revise{t}}) \, \d t + \sqrt{2 \cov(\mfl_t)} \, \d \wn{j}_t,
        \qquad j \in \range{1}{J}, \\
        \mfl_t &= \operatorname{Law}\left(\xnl{j}_t\right).
    \end{aligned}
    \right.
\end{equation}
Using a novel bootstrapping argument,
which is summarized in~\cite[p.\ 142]{ReviewChaintronII},
they then prove that for every $\varepsilon > 0$,
there exists $C$ independent of $J$ such that
\begin{equation}
    \label{eq:chaos}
    \forall J \in \nat^+, \qquad
    \sup_{t \in [0, T]} \expect \left[ \bigl\lvert \xn{j}_t - \xnl{j}_t \bigr\rvert^2 \right]
    \leq C J^{-\frac{1}{2} + \varepsilon}.
\end{equation}
Combining this bound with Wasserstein convergence estimates for empirical measures formed from i.i.d.\ samples~\cite{MR3383341},
they then deduce a convergence estimate of the form
\begin{equation}
    \label{eq:empirical_chaos}
    \forall J \in \nat^+, \qquad
    \expect \Bigl[ W_2(\emp_T, \mfl_T) \Bigr] \leq C J^{-\alpha},
\end{equation}
for an appropriate $\alpha > 0$.
The derivation of~\eqref{eq:empirical_chaos} from~\eqref{eq:chaos} is straightforward and independent of the interacting particle system considered.
In the terminology of~\cite{ReviewChaintronI},
estimate~\eqref{eq:chaos} establishes pointwise infinite-dimensional Wasserstein-$2$ chaos,
whereas estimate~\eqref{eq:empirical_chaos} implies pointwise Wasserstein-$2$ empirical chaos.

\paragraph{Contributions of this work}
The aim of this note is to revisit propagation of chaos for the interacting particle system~\eqref{eq:eks}.
For simplicity, we consider neither the modification proposed in~\cite{nusken2019note,MR4123680} nor gradient-free approximations,
but note that the approach we present generalizes to the ALDI sampler from~\cite{MR4123680} in a straightforward manner.
Our contributions are the following:
\begin{itemize}
    \item
        We generalize the work of~\cite{MR4234152} by relaxing the assumption that~$\phi$ is quadratic.
        The assumptions made on~$\phi$ in our main results are similar to those in~\cite{MR4123680}.

    \item
        Whereas~\eqref{eq:chaos}, in the terminology of~\cite{ReviewChaintronII},
        may be viewed as a \emph{pointwise} estimate for the Wasserstein-2 norm,
        the result we prove is a general \emph{pathwise} estimate for the Wasserstein-$p$ norm.

    \item
        We \revise{deploy} an approach based on appropriate stopping times,
        which \revise{originates from numerical analysis~\cite{MR1949404}} but is novel in the context of mean field limits and may prove useful for the analysis of other interacting particle systems.
        This leads to an estimate which \revise{is} optimal,
        in the sense that~\eqref{eq:chaos} holds with $\varepsilon = 0$.

    \item
        Finally, we prove a novel well-posedness result for the mean field dynamics~\eqref{eq:synchronous_coupling},
        under a local Lipschitz continuity assumption on~$\nabla \phi$.
\end{itemize}
The rest of this work is organized as follows.
Well-posedness results for the interacting particle system and its formal mean field limit are presented in~\cref{section:wellposed}.
We then present auxiliary results in~\cref{sec:aux},
and the main results in~\cref{sec:main}.
The appendices contain the proofs of well-posedness results (\cref{sec:well-posed})
and auxiliary results~(\cref{sec:proof_of_aux}).

\paragraph{Notation}
We let $\N :=\{0,1,2,3,\dots\}$
and $\N^+ :=\{1,2,3,\dots\}$.
For a matrix $X \in \R^{d \times d}$,
the notation~$\norm{X}_{\rm F}$ refers to the Frobenius norm.
The set of probability measures over $\real^d$ is denoted by~$\mathcal P(\real^d)$,
and the notation $\mathcal P_{p}(\real^d) \subset \mathcal P(\real^d)$ refers to the set of probability measures with finite moments up to order~$p$.
For a probability measure~$\mu \in\mathcal{P}_2(\R^d)$
the following notation is used to denote the mean and covariance under~$\mu$:
\begin{align}
    \label{eq:cov_mean_def}
    \mathcal{M}(\mu) = \int_{\R^d} x \, \mu(\d x),
    \qquad
    \mathcal{C}(\mu) = \int_{\R^d} \bigl(x - \mathcal M(\mu)\bigr) \otimes \bigl(x - \mathcal M(\mu)\bigr)\, \mu(\d x).
\end{align}
The notation $W_p(\placeholder, \placeholder)$ denotes the $p$-Wasserstein distance, see~\cite[Chapter 6]{MR2459454}.
For a probability measure $\mu \in \mathcal P(\real^d)$ and a function $f\colon \real^d \to \R$,
we use the short-hand notation
\[
    \mu [f] = \int_{E} f(x) \, \mu(\d x).
\]
By a slight abuse of notation,
we sometimes write $\mu[f(x)]$ instead of $\mu[f]$ for convenience.
For example, for a probability measure~$\mu \in \mathcal P_2(\R)$,
the notation $\mu\left[  x^2 \right]$ refers to the second raw moment of~$\mu$.
Throughout this note,
the notation~$\Omega$ refers to the sample space,
and $C$ refers to a constant whose exact value is irrelevant in the context and may change from occurrence to occurrence.
Furthermore, in expressions such as $\expect \bigl\lvert f(X) \bigr\rvert^p$,
it is always assumed that the exponent is inside the expectation;
otherwise we write $\left( \expect \bigl[ f(X) \bigr] \right)^p$.
Finally, for convenience,
we let $|x|_* = \max\{1, |x|\}$, so that~$|x|_*^a \leq |x|_*^b$ for all $a \leq b$ and all $x \in \real$.

\section{Assumption and well-posedness}
\label{section:wellposed}

Throughout this note,
we denote by $\mathcal A(\ell)$ for $\ell \geq 0$ the set of negative log-densities~$\phi$ that satisfy the following assumptions,
with that value of $\ell$.

\begin{assumption}
    \label{assumption:main}
    Without loss of generality,
    we assume that the function $\phi \colon \real^d \to \real$ is bounded from below by 1.
    We assume furthermore that $\phi \in \cont^2(\real^d)$ and
    that there are positive constants $\l,\u$ and a compact set $K$ such that
    the following inequalities are satisfied for all $x \in \real^d \setminus K$:
    \begin{subequations}
        \begin{align}
            \label{assump:function_bounds}
            \l |x|^{\ell+2} &\leq \phi(x) \le \u |x|^{\ell+2}, \\
            \label{assump:grad_bounds}
            \l |x|^{\ell+1} &\leq \lvert \nabla \phi(x) \rvert \le \u |x|^{\ell+1}, \\
            \label{assump:hessian_bounds}
            \l |x|^{\ell} \I_d &\preccurlyeq \hessian \phi(x) \preccurlyeq \u \I_d |x|^{\ell}.
        \end{align}
    \end{subequations}
\end{assumption}
\begin{remark}
    \label{remark:first}
    A few comments are in order.
    \begin{itemize}
        \item
            \Cref{assumption:main} implies that there are constants $\widetilde \l, \widetilde \u$
            such that
            \begin{equation}
                \label{eq:assumption_allx}
                \forall x \in \real^d, \qquad
                \begin{aligned}
                    \widetilde \l |x|_*^{\ell + 2} \leq \phi(x) &\leq \widetilde \u |x|_*^{\ell +2}, \qquad
                    \bigl\lvert \nabla \phi(x) \bigr\rvert &\le \widetilde \u |x|_*^{\ell + 1}, \qquad
                     \revise{%
                        \hessian \phi(x) \preccurlyeq \widetilde \u \I_d |x|_*^{\ell}.
                    }
                \end{aligned}
            \end{equation}
            Note that these inequalities are valid for all $x \in \real^d$, without a compact set excluded.
            We shall henceforth assume that these inequalities are satisfied with the same constants as in~\cref{assumption:main}.
        \item
            The last bound in~\eqref{eq:assumption_allx},
            together with the assumption that $\phi \in \cont^2(\real^d)$,
            implies that $\nabla \phi$ is locally Lipschitz continuous.
            \revise{%
                Specifically, it holds for some $z$ on the straight segment joining $x$ and $y$ that
                \begin{align}
                    \label{eq:assump-f:lip-growth-gradient}
                    \forall x, y \in \R^d, \qquad
                    \bigl\lvert \nabla\phi(x)- \nabla\phi(y) \bigr\rvert
                    \leq \norm{\hessian \phi(z)}\bigl\lvert x - y \bigr\rvert
                    \le \lipphi \left( 1 + |x|^{\ell} + |y|^{\ell} \right) |x-y|.
                \end{align}
            }

        \item
            When the target distribution~\eqref{eq:target_measure} is the Bayesian posterior associated to an inverse problem with a linear forward model, Gaussian noise and a Gaussian prior,
            the function $\phi$ is quadratic.
            In this case \cref{assumption:main} is indeed satisfied with $\ell = 0$.

        \item
            \revise{
                The assumption that $\phi \geq 1$ is made for convenience,
                but it is not required.
                This assumption is particularly useful for the proofs in~\cref{sec:well-posed},
                where it enables to use short Lyapunov functions of the form $x \mapsto \phi(x)^q$ for arbitrary $q > 0$,
                instead of more cumbersome expressions such as $(\phi - \min \phi + 1)^q$.
            }
    \end{itemize}
\end{remark}

Before presenting well-posedness results,
we describe the mathematical setting more precisely than in~\cref{sec:intro}.
Given independent standard Brownian motions $\bigl(\wn{j}\bigr)_{j\in\N^+}$ in~$\real^d$
and independent random variables $\bigl(\xn{j}_0\bigr)_{j\in\N^+}$ sampled from some $\mfl_0\in\mathcal P(\R^d)$,
we consider for each $J \in \nat^+$ the \revise{interacting} particle system
\begin{equation}
    \label{eq:system_rewritten}
    \xn{j}_t = \xn{j}_0 + \int_0^t b\left(\xn{j}_s, \emp_s\right) \, \d s +  \int_0^t \sigma\left(\xn{j}_s, \emp_s\right) \d \wn{j}_s,
    \qquad j\in \range{1}{J},
\end{equation}
where the drift $b\colon \real^d \times \mathcal P(\real^d) \to \real^d$
and diffusion $\sigma \colon \real^d \times \mathcal P(\real^d) \to \real^{d\times d}$
are given by
\begin{equation}
    \label{eq:eks-drift-and-diff}
    b(x, \mu) \coloneq - \cov(\mu) \nabla \phi(x),
    \qquad
    \sigma(x, \mu) \coloneq \sqrt{2\cov(\mu)}.
\end{equation}
Following the classical synchronous coupling approach pioneered by Sznitman~\cite{MR1108185},
we couple to~\eqref{eq:system_rewritten} the following system of i.i.d.\ mean-field McKean-Vlasov diffusions
\begin{align}
    \label{eq:coupling-system-mfl}
    \forall j\in \range{1}{J}, \qquad
    \left\{
    \begin{aligned}
        \xnl{j}_t &= \xn{j}_0 + \int_0^t b\left(\xnl{j}_s, \mfl_s\right) \, \d s +  \int_0^t \sigma\left(\xnl{j}_s, \mfl_s\right) \, \d \wn{j}_s,\\
        \mfl_t& = \operatorname{Law}\bigl(\xnl{j}_t\bigr),
    \end{aligned}
    \right.
\end{align}
driven by the same Brownian motions and with the same initial conditions as \eqref{eq:system_rewritten}.
The well-posedness of~\eqref{eq:system_rewritten} and~\eqref{eq:coupling-system-mfl}
follows from~\cref{proposition:moment_estimates} and~\cref{proposition:well-posedness} respectively,
stated hereafter and proved in~\cref{sec:well-posed}.

\begin{proposition}
    [Well-posedness for the interacting particle system]
    \label{proposition:moment_estimates}
    Assume that $\phi \in \mathcal A(\ell)$ for some $\ell \geq 0$ and that~$\mfl_0 \in \mathcal P_p(\real^d)$ for some $p \geq 2$.
    Then for any $J \in \nat^+$,
    the system of stochastic differential equations~\eqref{eq:system_rewritten}
    has a unique globally defined strong solution
    that is almost surely continuous.
    Furthermore, \revise{if $p \geq 4$} then for all $T > 0$ there is~$\kappa = \kappa(p) > 0$ independent of~$J$ such that
    \begin{align}
        \label{eq:moment_bounds_2}
        &\expect \left[\sup_{t \in [0, T]} \bigl\lvert \xn{j}_t \bigr\rvert^p \right]
        \leq \kappa.
    \end{align}
\end{proposition}
\begin{proof}
    The proof of well-posedness is essentially that of~\cite[Proposition 4.4]{MR4123680}.
    The moment estimate~\eqref{eq:moment_bounds_2} generalizes those from~\cite{MR4234152}.
    See~\cref{sec:well-posed} for details.
\end{proof}

\begin{proposition}
    [Well-posedness for the mean field dynamics]
    \label{proposition:well-posedness}
    Suppose that $\phi \in \mathcal A(\ell)$ \revise{for some~$\ell \geq 0$},
    that $\mfl_0\in\mathcal{P}_{p}(\R^d)$ \revise{for some~$p \geq 2 (\ell + 2)$},
    and that $\cov(\mfl_0) \succ 0$.
    Fix~$x_0 \sim \mfl_0$ and $T>0$.
    Then, there exists a unique strong solution~$\xl \in \cont([0,T], \R^d)$ to~\eqref{eq:mckean}
    such that~$\xl_0 = x_0$ and $t \mapsto \cov(\mfl_t)$ is continuous in $[0, T]$.
    Furthermore, the function $t \mapsto \cov(\mfl_t)$ is differentiable
    and there is~$\overline \kappa = \overline \kappa(p) > 0$ such that
    \begin{equation}
        \label{eq:moment_bound_mfl}
        \expect \left[ \sup_{t\in[0,T]}  \abs{\xl_{t}}^p \right] \leq \overline \kappa,
        \qquad \sup_{t \in [0, T]} \bigl\lVert \cov(\mfl_t)^{-1} \bigr\rVert_{\rm F} \leq \overline \kappa,
        \qquad \sup_{t \in [0, T]} \left\lVert \frac{\d \cov(\mfl_t)}{\d t} \right\rVert_{\rm F} \leq \overline \kappa,
    \end{equation}
\end{proposition}
\begin{proof}
    The proof uses a classical fixed-point approach,
    similar to that used in~\cite{carrillo2018analytical}.
    See~\cref{sec:well-posed} for details.
\end{proof}

\section{Auxiliary results}
\label{sec:aux}

The proof of the main results presented in~\cref{sec:main} relies on the \revise{moment bounds~\eqref{eq:moment_bounds_2} and~\eqref{eq:moment_bound_mfl}},
as well as the following auxiliary lemmas.

\begin{lemma}
    [Bound on the probability of large excursions]
    \label{lemma:small_set}
    Let $(Z_j)_{j\in\N^+}$ be a family of i.i.d.\ $\real$-valued random variables such that $\expect  \bigl[\abs*{Z_1}^{r}\bigr]  < \infty$
    \revise{for some $r \geq 2$}.
    Then for all $R > \expect \bigl[\abs*{Z_1}\bigr]$,
    there exists a constant $C>0$ such that
    \[
        \forall J \in \N^+, \qquad
        \proba \left[  \frac{1}{J} \sum_{j=1}^{J} Z_j \geq R \right]
        \leq C J^{-\frac{r}{2}}.
    \]
\end{lemma}
\begin{proof}
    This follows from a generalization of Chebychev's inequality~\cite[Exercise 3.21]{MR3674428}.
    Let
    \[
        X = \frac{1}{J} \sum_{j=1}^{J} Z_j.
    \]
    By the Marcinkiewicz--Zygmund \revise{and Jensen} inequalities,
    it holds that
    $\expect \bigl\lvert X - \expect [X] \bigr\rvert^{r} \leq C_{\rm MZ}(r) J^{-\frac{r}{2}} \expect \bigl\lvert Z_1 - \expect [Z_1] \bigr\rvert^r $.
    Therefore,
    using the Markov inequality,
    we deduce that
    \begin{align*}
        \proba \left[ X  \ge R \right]
        \leq \proba \Bigl[ \bigl\lvert X - \expect [X]  \bigr\rvert^r  \ge \bigl(R - \expect [X]\bigr)^r \Bigr]
        &\leq \expect \left[ \frac{\bigl\lvert X - \expect[X] \bigr\rvert^r}
        {\bigl(R - \expect [X]\bigr)^r} \right]
        \leq \frac{C J^{-\frac{r}{2}}}{\bigl(R - \expect [X]\bigr)^r} ,
    \end{align*}
    which concludes the proof.
\end{proof}

\begin{lemma}
    [Wasserstein stability estimates]
    \label{lemma:wasserstein_stability_estimates}
    For all $(\mu, \nu) \in \mathcal P_{2}\bigl(\real^d\bigr) \times \mathcal P_{2} \bigl(\real^d\bigr)$,
    it holds that
    \begin{subequations}
        \begin{align}
            \label{eq:stab_wcov_simple}
            \Bigl\lVert \cov(\mu) - \cov(\nu) \Bigr\rVert_{\rm F}
            &\le 2\Bigl( W_2(\mu, \delta_0) + W_2(\nu, \delta_0) \Bigr) W_2(\mu, \nu), \\
            \label{eq:wmean-wcov-emp-local-lip}
            \qquad
            \norm{
            \sqrt{\cov(\mu)} - \sqrt{ \cov (\nu) }}_{\rm F}
            &\le \sqrt{2} \, W_2(\mu, \nu).
        \end{align}
    \end{subequations}
\end{lemma}
\begin{proof}
    See~\cref{sub:proof_of_stab}.
\end{proof}

\begin{lemma}
    [Convergence of the empirical covariance for i.i.d. samples]
    \label{lemma:convergence_covariance_iid}
    \revise{For $p \geq 2$ and} all~$\mu \in \mathcal P_{2p}(\R^d)$,
    there is $C$ depending only on~$p$ and the $2p$-th moment of~$\mu$ such that
    for all $J \in \N^+$,
    \begin{equation}
        \label{eq:iid_convergence_cov}
        \expect \left\lVert \cov(\empmfl) - \cov(\mu) \right\rVert_{\rm F}^p
        \leq C J^{-\frac{p}{2}},
        \qquad
        \empmfl := \frac{1}{J} \sum_{j=1}^{J} \delta_{\xnl{j}},
        \qquad
        \left\{ \xnl{j} \right\}_{j \in \N} \stackrel{\rm{i.i.d.}}{\sim} \mu.
    \end{equation}
    Furthermore, for all~$\mu \in \mathcal P_{2p}(\R^d)$ satisfying $\cov(\mu) \succcurlyeq \eta \I_d \succ 0$,
    there is $C$ depending only on~$(p, \eta)$ and the $2p$-th moment of~$\mu$ such that
    for all $J \in \N^+$,
    \begin{equation}
        \label{eq:iid_convergence_sqrt}
        \expect \left\lVert \sqrt{\cov(\empmfl)} - \sqrt{\cov(\mu)} \right\rVert_{\rm F}^p
        \leq C J^{-\frac{p}{2}}.
    \end{equation}
\end{lemma}
\begin{proof}
    This is essentially~\cite[Lemma 3]{MR4234152}.
    We include a short proof in~\cref{sub:proof_of_convergence_covarianec_iid} for the reader's convenience.
\end{proof}

\section{Main results}
\label{sec:main}

We first present, in~\cref{sub:lip_thm}, a sharp quantitative result in the setting where $\phi \in \mathcal A(0)$,
in which case $\nabla \phi$ is globally Lipschitz continuous.
Then, in~\cref{sub:non_lip_thm},
we extend the result to the setting where $\phi \in \mathcal A(\ell)$ for~$\ell > 0$.
While the methodology used in the proof of~\cref{theorem:mfl} is rather general
and potentially applicable to other interacting particle systems,
the approach used in the proof of~\cref{theorem:mfl_nonlip} is more technical and tailored to the interacting particle system~\eqref{eq:system_rewritten}.
Finally, in~\cref{sub:corollary},
we present a corollary of these theorems with applications in sampling.

\subsection{Sharp propagation of chaos for globally Lispchitz \texorpdfstring{$\nabla \phi$}{setting}}
\label{sub:lip_thm}
\begin{theorem}
    \label{theorem:mfl}
    Suppose that~$\phi \in \mathcal A(0)$,
    and consider the systems~\eqref{eq:system_rewritten} and \eqref{eq:coupling-system-mfl} with the coefficients given in~\eqref{eq:eks-drift-and-diff}.
    Assume that $\mfl_0 \in \mathcal P_{q}(\real^d)$,
    for some $q \geq 6$, and that $\cov(\mfl_0) \succ 0$.
    Then for all~$p \in [2, \frac{q}{3}]$,
    there is~$C > 0$ independent of~$J$ such that
    \begin{equation}
        \label{eq:statement_main_theorem}
        \forall J\in\N^+, \qquad
        \forall j \in \range{1}{J}, \qquad
        \expect \left[ \sup_{t\in[0,T]} \abs[\Big]{\xn{j}_t-\xnl{j}_t}^p \right]
        \le C J^{- \min\left\{ \frac{p}{2}, \frac{q - p}{2 \sqrt{q}}, \frac{q-p}{2p}, \frac{q-p}{6} \right\}}.
    \end{equation}
\end{theorem}
\begin{proof}
    Fix $p \in [2, \frac{q}{3}]$, and fix $\r \in [p, \frac{q}{3}]$ to be determined later.
    Fix also $R \in (0, \infty)$ such that
    \begin{equation}
        \label{eq:definition_R}
        \left(\frac{R}{2}\right)^{\r} >  \expect \left[ \sup_{t \in [0, T]} \left\lvert \xnl{j}_t \right\rvert^{\r} \right].
    \end{equation}
    By~\cref{proposition:well-posedness},
    the right-hand side \revise{of this equation} is indeed finite.
    Consider the stopping times
    \begin{subequations}
    \begin{align}
        \label{eq:stopping_times}
        \tau_J(R) &= \inf \left\{ t \geq 0 : \frac{1}{J} \sum_{j=1}^{J} \left\lvert \xn{j}_t \right\rvert^{\r} \geq R^{\r} \right\}
        = \inf \Bigl\{ t \geq 0 : W_{\r}(\emp_t, \delta_0) \geq R \Bigr\} , \\
        \overline \tau_J(R) &= \inf \left\{ t \geq 0 : \frac{1}{J} \sum_{j=1}^{J} \left\lvert \xnl{j}_t \right\rvert^{\r} \geq R^{\r} \right\}
        = \inf \Bigl\{ t \geq 0 : W_{\r}(\empmfl_t, \delta_0) \geq R \Bigr\},
    \end{align}
    \end{subequations}
    and let~$\theta_J(R) := \min\Bigl\{\tau_J(R), \overline \tau_J(R)\Bigr\}$.
    Since $R$ will be fixed until the end of the proof,
    we omit this dependence in the notation.
    Fix $j \in \range{1}{J}$.
    By Hölder's inequality,
    it holds that
    \begin{align}
        \notag
        \expect \left[ \sup_{t\in[0,T]} \abs[\Big]{\xn{j}_t - \xnl{j}_t}^p \right]
        &= \expect \left[ \sup_{t\in[0,T]} \abs[\Big]{\xn{j}_t - \xnl{j}_t}^p \mathsf 1_{\{\theta_J > T\}} \right]
        + \expect \left[ \sup_{t\in[0,T]} \abs[\Big]{\xn{j}_t - \xnl{j}_t}^p \mathsf 1_{\{\theta_J \leq T\}} \right] \\
        \label{eq:main_equation}
        &\leq \expect \left[ \sup_{t\in[0,T]} \abs[\Big]{\xn{j}_{t \wedge \theta_J} - \xnl{j}_{t \wedge \theta_J}}^p \right]
        + \left(\expect \left[ \sup_{t\in[0,T]} \abs[\Big]{\xn{j}_t - \xnl{j}_t}^{q} \right]\right)^{\frac{p}{q}}
        \Bigl( \proba [\theta_J \leq T]\Bigr)^{\frac{q-p}{q}}.
    \end{align}
    We bound the two terms on the right-hand side separately,
    and then conclude the proof.

    \paragraph{Step A. Bounding the first term in~\eqref{eq:main_equation}}
    Since $\norm{Y}_{L^{a}(\Omega)} \leq \norm{Y}_{L^{b}(\Omega)}$ for any random variable $Y$ and any $a \leq b$,
    it holds that
    \begin{equation}
        \label{eq:sznitman_p_to_r}
        \expect \left[ \sup_{t\in[0,T]} \abs[\Big]{\xn{j}_{t \wedge \theta_J} - \xnl{j}_{t \wedge \theta_J}}^p \right]
        \leq
        \left(\expect \left[ \sup_{t\in[0,T]} \abs[\Big]{\xn{j}_{t \wedge \theta_J} - \xnl{j}_{t \wedge \theta_J}}^r \right] \right)^{\frac{p}{r}}.
    \end{equation}
    Bounding the right-hand side of this inequality is not simpler than bounding the left-hand side directly.
    However, the bound obtained with $r \geq p$ will be useful to bound the second term in~\eqref{eq:main_equation}.
    In order to bound the right-hand side of~\eqref{eq:sznitman_p_to_r},
    we adapt Sznitman's classical argument.
    We have
    \begin{align*}
        \frac{1}{2^{{\r}-1}}\left\lvert \xn{j}_{t \wedge \theta_J} - \xnl{j}_{t \wedge \theta_J} \right\rvert^{\r}
        \leq
    & \left\lvert \int_{0}^{t \wedge \theta_J} b\left(\xn{j}_s, \emp_{s}\right) - b\left(\xnl{j}_s, \mfl_{s} \right) \, \d s \right\rvert^{\r} \\
    &+  \left\lvert \int_{0}^{t \wedge \theta_J} \sigma\left(\xn{j}_s, \emp_{s}\right) - \sigma\left(\xnl{j}_s, \mfl_{s} \right) \, \d \wn{j}_s \right\rvert^{\r}.
    \end{align*}
    Let us introduce the martingale
    \[
        M^j_t = \int_{0}^{t} \sigma\left(\xn{j}_s, \emp_{s}\right) - \sigma\left(\xnl{j}_s, \mfl_{s} \right) \, \d \wn{j}_s.
    \]
    By Doob's optional stopping theorem~\cite[Theorem 3.3]{MR2380366},
    see also~\cite[Equation 2.29, p.285]{MR0838085},
    the process $(M^j_{t \wedge \theta_J})_{t\geq 0}$ is a martingale,
    with a quadratic variation process given by~$(\langle M^j \rangle_{t \wedge \theta_J})_{t \geq 0}$,
    where $\langle M^j \rangle$ is the quadratic variation process of~$M^j$.
    Therefore, by the Burkholder--Davis--Gundy inequality,
    we have for all $t \in [0, T]$ that
    \begin{align}
        \notag
        \expect \left[ \sup_{s \in [0,t]} \Bigl\lvert \xn{j}_{s \wedge \theta_J} - \xnl{j}_{s \wedge \theta_J} \Bigr\rvert^{\r} \right]
    &\leq (2T)^{{\r}-1} \expect \int_{0}^{t \wedge \theta_J} \left\lvert b\left(\xn{j}_s, \emp_{s}\right) - b\left(\xnl{j}_s, \mfl_{s} \right) \right\rvert^{\r} \, \d s \\
    \label{eq:main_first_term}
    &\qquad + C_{\rm BDG} 2^{{\r}-1}T^{\frac{\r}{2} - 1} \expect \int_{0}^{t \wedge \theta_J} \left\lVert \sigma\left(\xn{j}_s, \emp_{s}\right) - \sigma\left(\xnl{j}_s, \mfl_{s} \right) \right\rVert_{\rm F}^{\r} \, \d s.
    \end{align}
    From the triangle inequality,
    it holds that
    \begin{align}
        \notag
        \frac{1}{2^{\r-1}}\expect \int_{0}^{t \wedge \theta_J} \left\lvert b\left(\xn{j}_s, \emp_{s}\right) - b\left(\xnl{j}_s, \mfl_{s} \right) \right\rvert^{\r} \, \d s
        &\leq \int_{0}^{t} \expect \left\lvert b\left(\xn{j}_{s \wedge \theta_J}, \emp_{s \wedge \theta_J}\right) - b\left(\xnl{j}_{s \wedge \theta_J}, \empmfl_{s \wedge \theta_J}\right) \right\rvert^{\r} \, \d s \\
    \label{eq:decomposition_drift}
    &\qquad + \int_{0}^{t} \expect \left\lvert b\left(\xnl{j}_s, \empmfl_{s}\right) - b\left(\xnl{j}_s, \mfl_{s}\right) \right\rvert^{\r} \, \d s.
    \end{align}
    Similarly, for the diffusion term,
    we have
    \begin{align}
        \notag
        \frac{1}{2^{\r-1}}\expect \int_{0}^{t \wedge \theta_J} \left\lVert \sigma\left(\xn{j}_s, \emp_{s}\right) - \sigma\left(\xnl{j}_s, \mfl_{s} \right) \right\rVert_{\rm F}^{\r} \, \d s
    &\leq \int_{0}^{t} \expect \left\lVert \sigma\left(\xn{j}_{s \wedge \theta_J}, \emp_{s \wedge \theta_J}\right) - \sigma\left(\xnl{j}_{s \wedge \theta_J}, \empmfl_{s \wedge \theta_J}\right) \right\rVert_{\rm F}^{\r} \, \d s \\
    \label{eq:decomposition_diffusion}
    &\qquad + \int_{0}^{t} \expect \left\lVert \sigma\left(\xnl{j}_s, \empmfl_{s}\right) - \sigma\left(\xnl{j}_s, \mfl_{s}\right) \right\rVert_{\rm F}^{\r} \, \d s.
    \end{align}
    Next, we bound the terms on the right-hand side of~\eqref{eq:decomposition_drift} and~\eqref{eq:decomposition_diffusion}.

    \paragraph{\textit{Step A.1. Bounding the first term in~\eqref{eq:decomposition_drift}}}
    By~\eqref{eq:eks-drift-and-diff} and the triangle inequality,
    it holds that
    \begin{align*}
        \expect \left\lvert b\left(\xn{j}_{s \wedge \theta_J}, \emp_{s \wedge \theta_J}\right) - b\left(\xnl{j}_{s \wedge \theta_J}, \empmfl_{s \wedge \theta_J}\right) \right\rvert^{\r}
        &\leq
        2^{\r-1} \expect \left\lvert \Bigl( \cov \left(\emp_{s \wedge \theta_J}\right) - \cov \left(\empmfl_{s \wedge \theta_J}\right) \Bigr) \nabla \phi\left(\xn{j}_{s \wedge \theta_J}\right)  \right\rvert^{\r} \\
        &\qquad
        + 2^{\r-1} \expect \left\lvert  \cov \left(\empmfl_{s \wedge \theta_J}\right) \left( \nabla \phi\left(\xn{j}_{s \wedge \theta_J}\right) - \nabla \phi\left(\xnl{j}_{s \wedge \theta_J}\right) \right)  \right\rvert^{\r}.
    \end{align*}
    By~\eqref{eq:stab_wcov_simple} in~\cref{lemma:wasserstein_stability_estimates},
    we obtain
    \[
        \Bigl\lVert \cov \left(\emp_{s \wedge \theta_J}\right) - \cov \left(\empmfl_{s \wedge \theta_J} \right) \Bigr\rVert_{\rm F}
        \leq 2 \Bigl( W_2 \left(\emp_{s \wedge \theta_J}, \delta_0 \right) + W_2 \left(\empmfl_{s \wedge \theta_J}, \delta_0 \right) \Bigr) W_2 \left(\emp_{s \wedge \theta_J}, \empmfl_{s \wedge \theta_J} \right)
        \leq 4 R \, W_2 \left(\emp_{s \wedge \theta_J}, \empmfl_{s \wedge \theta_J} \right),
    \]
    where we used the definition of the stopping times in the second inequality.
    Therefore,
    we deduce that
    \begin{align*}
        \expect \left\lvert \Bigl( \cov \left(\emp_{s \wedge \theta_J}\right) - \cov \left(\empmfl_{s \wedge \theta_J}\right) \Bigr) \nabla \phi\left(\xn{j}_{s \wedge \theta_J}\right)  \right\rvert^{\r}
        &\leq (4R)^\r \expect \left\lvert W_2\left(\emp_{s \wedge \theta_J}, \empmfl_{s \wedge \theta_J} \right) \nabla \phi\left(\xn{j}_{s \wedge \theta_J}\right)  \right\rvert^{\r} \\
        &=  (4R)^\r \expect \left[ W_2\left(\emp_{s \wedge \theta_J}, \empmfl_{s \wedge \theta_J} \right)^{\r} \frac{1}{J} \sum_{k=1}^{J} \Bigl\lvert  \nabla \phi\left(\xn{k}_{s \wedge \theta_J}\right)  \Bigr\rvert^{\r} \right],
    \end{align*}
    where we used exchangeability in the second line.
    By the assumption~\eqref{assump:grad_bounds} of linear growth of $\nabla \phi$
    and the definition of~$\theta_J$,
    this leads to
    \begin{align}
        \notag
        \expect \left\lvert \Bigl( \cov \left(\emp_{s \wedge \theta_J}\right) - \cov \left(\empmfl_{s \wedge \theta_J}\right) \Bigr) \nabla \phi\left(\xn{j}_{s \wedge \theta_J}\right)  \right\rvert^{\r}
        \notag
        &\leq C \expect \left[ W_2\left(\emp_{s \wedge \theta_J}, \empmfl_{s \wedge \theta_J} \right)^{\r}  \left( \frac{1}{J} \sum_{k=1}^{J} \left\lvert \xn{k}_{s \wedge \theta_J} \right\rvert_*^{\r} \right) \right] \\
        \notag
        &\leq C \expect \left[ W_2\left(\emp_{s \wedge \theta_J}, \empmfl_{s \wedge \theta_J} \right)^{\r}  \Bigl( 1 + W_{\r} (\emp_{s \wedge \theta_J}, \delta_0)^{\r} \Bigr) \right] \\
        \label{eq:proof_term_1_bound_1}
        &\leq C \expect \left[ W_2\left(\emp_{s \wedge \theta_J}, \empmfl_{s \wedge \theta_J} \right)^{\r}   \right].
    \end{align}
    The constant $C$ on the right-hand side depends on $R$,
    but since $R$ is fixed independently of~$J$,
    we omit this dependence.
    On the other hand,
    by definition of the stopping time $\theta_J$,
    and by the inequality $\lVert \cov(\mu) \rVert_{\rm F} \leq W_2(\mu, \delta_0)^2$ which holds for all~$\mu \in \mathcal P(\real^d)$,
    it holds that
    \begin{align}
        \nonumber
        \expect \left\lvert  \cov \left(\empmfl_{s \wedge \theta_J}\right) \left( \nabla \phi\left(\xn{j}_{s \wedge \theta_J}\right) - \nabla \phi\left(\xnl{j}_{s \wedge \theta_J}\right) \right)  \right\rvert^{\r}
        &\leq R^{2 \r} \expect \left| \nabla \phi\left(\xn{j}_{s \wedge \theta_J}\right) - \nabla \phi\left(\xnl{j}_{s \wedge \theta_J} \right) \right\rvert^{\r}, \\
        \label{eq:proof_term_1_bound_2}
        &\leq R^{2 \r} \lipphi^r \expect \left| \xn{j}_{s \wedge \theta_J} - \xnl{j}_{s \wedge \theta_J}  \right\rvert^{\r}.
    \end{align}
    where we used the assumption of Lipschitz continuity on~$\nabla \phi$ in~\cref{assumption:main}.
    In view of the inequalities~\eqref{eq:proof_term_1_bound_1} and~\eqref{eq:proof_term_1_bound_2},
    and of the bound
    \begin{equation}
        \label{eq:bound_wasserstein}
        \expect \left[W_{\r}\left(\emp_{s \wedge \theta_J}, \empmfl_{s \wedge \theta_J}\right)^{\r}\right]
        \leq \expect \left[ \frac{1}{J}\sum_{k=1}^{J} \left\lvert \xn{k}_{s \wedge \theta_J} - \xnl{k}_{s \wedge \theta_J} \right\rvert^{\r} \right]
        = \expect \left\lvert \xn{j}_{s \wedge \theta_J} - \xnl{j}_{s \wedge \theta_J} \right\rvert^{\r},
    \end{equation}
    which holds by definition of the Wasserstein distance and exchangeability,
    we deduce that
    \[
        \expect \left\lvert b\left(\xn{j}_{s \wedge \theta_J}, \emp_{s \wedge \theta_J}\right)
        - b\left(\xnl{j}_{s \wedge \theta_J}, \empmfl_{s \wedge \theta_J}\right) \right\rvert^{\r}
        \leq C \expect \left\lvert \xn{j}_{s \wedge \theta_J} - \xnl{j}_{s \wedge \theta_J} \right\rvert^{\r}.
    \]

    \paragraph{\textit{Step A.2. Bounding the second term in~\eqref{eq:decomposition_drift}}}
    For this term,
    we have
    \begin{align*}
        \expect \left\lvert b\left(\xnl{j}_s, \empmfl_{s}\right) - b\left(\xnl{j}_s, \mfl_{s}\right) \right\rvert^{\r}
        &=
        \expect \abs[\Big]{ \Bigl( \cov(\empmfl_s) - \cov(\mfl_s) \Bigr) \nabla \phi\bigl(\xnl{j}_s\bigr)}^{\r} \\
        &\leq C \expect \left[ \Bigl\lVert \cov(\empmfl_s) - \cov(\mfl_s) \Bigr\rVert_{\rm F}^{\r} \left( 1 + \bigl\lvert \xnl{j}_s \bigr\rvert^{\r} \right) \right]    \\
        &\leq C \left( \expect \Bigl\lVert \cov(\empmfl_s) - \cov(\mfl_s) \Bigr\rVert_{\rm F}^{\frac{3 \r}{2}} \right)^{\frac{2}{3}}
        \left(\expect \left[ 1 + \bigl\lvert \xnl{j}_s \bigr\rvert^{3\r} \right] \right)^{\frac{1}{3}},
    \end{align*}
    where we used~\eqref{assump:grad_bounds} in~\cref{assumption:main} and H\"older's inequality.
    Using the moment bound in~\cref{proposition:well-posedness} and then using \cref{lemma:convergence_covariance_iid},
    noting that $\mfl_0 \in \mathcal P_{3 \r}(\real^d)$ since $3r \leq q$,
    we deduce that
    \[
        \expect \left\lvert b\left(\xnl{j}_s, \empmfl_{s}\right) - b\left(\xnl{j}_s, \mfl_{s}\right) \right\rvert^{\r}
        \leq C J^{-\frac{\r}{2}}.
    \]

    \paragraph{\textit{Step A.3. Bounding the first term in~\eqref{eq:decomposition_diffusion}}}
    For the first diffusion term,
    we have that
    \begin{align*}
        \left\lVert \sigma\left(\xn{j}_{s \wedge \theta_J}, \mu^J_{s \wedge \theta_J}\right) - \sigma\left(\xnl{j}_{s \wedge \theta_J}, \empmfl_{s \wedge \theta_J}\right) \right\rVert_{\rm F}^{\r}
             &= \revise{2^{\frac{r}{2}}} \biggl\lVert \sqrt{\cov\left(\mu^J_{s \wedge \theta_J} \right)} - \sqrt{\cov\left(\empmfl_{s \wedge \theta_J} \right) } \biggr\rVert_{\rm F}^{\r}.
    \end{align*}
    Using~\cref{lemma:wasserstein_stability_estimates}
    together with the bound~\eqref{eq:bound_wasserstein},
    we obtain
    \[
        \expect \left\lVert \sigma\left(\xn{j}_{s \wedge \theta_J}, \mu^J_{s \wedge \theta_J}\right) - \sigma\left(\xnl{j}_{s \wedge \theta_J}, \empmfl_{s \wedge \theta_J}\right) \right\rVert_{\rm F}^{\r}
        \leq \revise{2^r} \expect \left\lvert \xn{j}_{s \wedge \theta_J} - \xnl{j}_{s \wedge \theta_J} \right\rvert^{\r}.
    \]

    \paragraph{\textit{Step A.4. Bounding the second term in~\eqref{eq:decomposition_diffusion}}}
    By~\cref{proposition:well-posedness},
    there is $\eta > 0$ such that $\cov(\mfl_t) \succcurlyeq \eta \I_d$ for all~$t \in [0, T]$.
    Thus, it follows from~\cref{lemma:convergence_covariance_iid} that
    \[
        \expect \left\lVert \sigma\left(\xnl{j}_{s}, \empmfl_{s}\right) - \sigma\left(\xnl{j}_{s}, \mfl_{s}\right) \right\rVert_{\rm F}^{\r}
        = 2^{\frac{\r}{2}} \expect \left\lVert \sqrt{\cov \left(\empmfl_{s}\right)} - \sqrt{\cov \left(\mfl_s \right)} \right\rVert_{\rm F}^{\r}
        \leq C J^{-\frac{\r}{2}}.
    \]
    \paragraph{\textit{Step A.5. Concluding part A}}
    Combining the bounds on all terms,
    we finally obtain from~\eqref{eq:main_first_term} that
    \[
        \expect \left[ \sup_{s\in[0,t]} \abs[\Big]{\xn{j}_{s \wedge \theta_J} - \xnl{j}_{s \wedge \theta_J}}^{\r} \right]
        \leq C J^{- \frac{\r}{2}} +
        C \int_{0}^{t} \expect \left[ \sup_{u\in[0,s]} \Bigl\lvert \xn{j}_{u \wedge \theta_J} - \xnl{j}_{u \wedge \theta_J} \Bigr\rvert^{\r} \right] \, \d s.
    \]
    By Gr\"onwall's inequality,
    this implies that
    \begin{equation}
        \label{eq:gronwall}
        \expect \left[ \sup_{t\in[0,T]} \Bigl\lvert \xn{j}_{t \wedge \theta_J} - \xnl{j}_{t \wedge \theta_J} \Bigr\rvert^{\r} \right]
        \leq C J^{- \frac{\r}{2}}.
    \end{equation}
    \paragraph{Step B. Bounding the second term in~\eqref{eq:main_equation}}
    We have by~\cref{proposition:well-posedness,proposition:moment_estimates} that
    \begin{equation}
        \label{eq:moments_first_proof}
        \expect \left[ \sup_{t\in[0,T]} \abs[\big]{\xn{j}_t - \xnl{j}_t}^{q} \right]
        \leq 2^{q-1} \left( \expect \left[ \sup_{t\in[0,T]} \abs[\big]{\xn{j}_t}^{q} \right]
        + \expect \left[ \sup_{t\in[0,T]} \abs[\big]{\xnl{j}_t}^{q} \right] \right)
        \leq 2^{q-1} \bigl(\kappa(q) + \overline \kappa(q)\bigr).
    \end{equation}
    In order to complete the proof of the theorem,
    it remains to bound the probability $\proba \left[\theta_J \leq T \right]$,
    which can be achieved by noticing that
    \[
        \proba \bigl[\theta_J \leq T\bigr]
        =
        \proba \bigl[\tau_J \leq T < \overline \tau_J \bigr]
        + \proba \bigl[\overline \tau_J \leq T\bigr].
    \]
    Using the almost sure continuity of the solution to the interacting particle system,
    together with the triangle inequality,
    we bound the first probability as follows:
    \begin{align*}
        \proba [\tau_J \leq T < \overline \tau_J]
            &\leq \proba \left[ \sup_{t \in [0, T]} W_{\r}(\mu^J_{t \wedge \theta_J}, \delta_0) = R \right] \\
            &\leq
            \proba \left[ \sup_{t \in [0, T]} W_{\r}(\mu^J_{t \wedge \theta_J}, \empmfl_{t \wedge \theta_J}) + \sup_{t \in [0, T]} W_{\r}(\empmfl_{t \wedge \theta_J}, \delta_0) \geq R \right] \\
            &\leq
            \proba \left[\sup_{t \in [0, T]} W_{\r}(\mu^J_{t \wedge \theta_J}, \empmfl_{t \wedge \theta_J}) \geq \frac{R}{2} \right]
            + \proba \left[\sup_{t \in [0, T]} W_{\r}(\empmfl_{t \wedge \theta_J}, \delta_0) \geq \frac{R}{2} \right],
    \end{align*}
    where we used that $\proba[A +  B \geq k] \leq \proba [A \geq k/2] + \proba [B \geq k/2]$
    for any two real-valued random variables~$A$ and~$B$,
    because $\{A + B \geq k\} \subset \{A \geq k/2\} \cup \{B \geq k/2\}$.
    The probability $\proba [\theta_J \leq T]$ can then be bounded as follows:
    \begin{align*}
        \proba [\theta_J \leq T]
            &=
            \proba \left[\tau_J \leq T < \overline \tau_J \right]
            + \proba \left[\overline \tau_J \leq T\right] \\
            &\leq \proba \left[\sup_{t \in [0, T]} W_{\r}(\mu^J_{t \wedge \theta_J}, \empmfl_{t \wedge \theta_J}) \geq \frac{R}{2} \right]
            + 2 \proba \left[\sup_{t \in [0, T]} W_{\r}(\empmfl_{t}, \delta_0) \geq \frac{R}{2} \right].
    \end{align*}
    Using~\eqref{eq:bound_wasserstein} and Markov's inequality for the first term,
    together with the inequality $\sup \sum \leq \sum \sup$,
    we then obtain
    \begin{equation}
        \label{eq:auxiliary_bound_proba}
        \proba [\theta_J \leq T]
        \leq \frac{2^{\r}}{R^{\r}} \expect \left[ \frac{1}{J}\sum_{j=1}^J \sup_{t\in[0,T]} \Bigl\lvert \xn{j}_{t\wedge \theta_J} - \xnl{j}_{t\wedge \theta_J} \Bigr\rvert^{\r} \right]
        + 2\proba \left[\frac{1}{J}\sum_{j=1}^J \sup_{t\in[0,T]} \Bigl\lvert \xnl{j}_t \Bigr\rvert^{\r} \geq \frac{R^{\r}}{2^{\r}} \right].
    \end{equation}
    The first term is bounded from above by~$CJ^{-\frac{\r}{2}}$ by exchangeability and~\eqref{eq:gronwall}.
    In order to bound the second term,
    let us introduce the i.i.d.\ random variables
    \[
        Z_j = \sup_{t \in [0, T]} \left\lvert \xnl{j}_t \right\rvert^{\r}, \qquad j \in \range{1}{J}.
    \]
    By the moment bounds in~\cref{proposition:well-posedness} and the assumption that $\mfl_0 \in \mathcal P_{q}(\real^d)$,
    the random variable $Z_1$ has finite moments up to order $\frac{q}{\r} \geq 3$.
    Furthermore, by definition~\eqref{eq:definition_R} of $R$,
    it holds that $\expect [Z_1] < \frac{R^{\r}}{2^{\r}}$.
    Thus,
    it follows from~\cref{lemma:small_set} that
    there is a constant $C$ independent of $J$ such that
    \begin{equation}
        \label{eq:bound_second_big_term}
        \proba \left[\frac{1}{J}\sum_{j=1}^J \sup_{t\in[0,T]} \Bigl\lvert \xnl{j}_t \Bigr\rvert^{\r} \geq \frac{R^{\r}}{2^{\r}} \right]
        =
        \proba \left[\frac{1}{J}\sum_{j=1}^J Z_j \geq \frac{R^{\r}}{2^{\r}} \right]
        \leq C J^{-\frac{q}{2 \r}}.
    \end{equation}

    \paragraph{Step C. Concluding the proof}
    Substituting the bounds~\eqref{eq:sznitman_p_to_r},~\eqref{eq:gronwall},~\eqref{eq:moments_first_proof},~\eqref{eq:auxiliary_bound_proba} and~\eqref{eq:bound_second_big_term} into~\eqref{eq:main_equation},
    we finally obtain that
    \begin{align*}
        \expect \left[ \sup_{t\in[0,T]} \abs[\Big]{\xn{j}_t - \xnl{j}_t}^p \right]
        &\revise{\leq \left( C J^{- \frac{r}{2}} \right)^{\frac{p}{r}} + \Bigl(2^{q-1} \bigl(\kappa(q) + \overline \kappa(q)\bigr)\Bigr)^{\frac{p}{q}} \left(C J^{-\frac{r}{2}} + C J^{-\frac{q}{2 \r}}\right)^{\frac{q-p}{q}}} \\
        &\leq C \left(J^{-\frac{p}{2}} + J^{-\frac{r(q-p)}{2q}} + J^{-\frac{q - p}{2r}} \right).
    \end{align*}
    This inequality is true for any $r \in [p, \frac{q}{3}]$,
    with a constant $C$ depending on $r$.
    The best estimate is obtained when the exponents of the second and third terms are as close to equal as possible,
    that is to say when $r = \max\bigl\{p,\min\{\sqrt{q}, \frac{q}{3}\}\bigr\}$,
    which leads to the estimate~\eqref{eq:statement_main_theorem} and concludes the proof.
\end{proof}

\begin{remark}
    \label{remark:2}
    A few comments are in order.
    \begin{itemize}
        \item
            If $\mfl_0$ has sufficiently many moments,
            then~\eqref{eq:statement_main_theorem} recovers the optimal convergence rate~$J^{-\frac{p}{2}}$,
            which is obtained in the classical setting
            with globally Lipschitz coefficients,
            see~\cite[Theorem 3.1]{ReviewChaintronII}.
        \item
            In the proof of~\cref{theorem:mfl},
            the probability $\proba [\tau_J(R) \leq T]$ was bounded from above in terms of $\proba [\overline \tau_J (\frac{R}{2}) \leq T]$
            and an appropriate distance between the stopped particle systems.
            The probability $\proba [\overline \tau_J (\frac{R}{2}) \leq T]$ is simple to bound directly,
            because the synchronously coupled mean field particles are independent and identically distributed.

        \item
            In order to obtain an estimate with a scaling that is optimal in~$J$,
            it was crucial to first prove in~\eqref{eq:gronwall} a propagation of chaos result for the stopped particle system
            in a metric $L^r$ with $r$ larger than the value $p$ in the final $L^p$ estimate~\eqref{eq:statement_main_theorem}.

        \item
            For $\phi \in \mathcal A(\ell)$ with $\ell > 0$,
            the proof presented above does not go through.
            The issue in this case is to obtain a bound similar to~\eqref{eq:proof_term_1_bound_2}.
            It is still possible, however, to prove sharp propagation of chaos in this case using a slightly different and more technical approach,
            which relies explicitly on the lower bound in~\eqref{assump:hessian_bounds} on the Hessian of the function $\phi$;
            see~\cref{theorem:mfl_nonlip}.

        \item
            \revise{%
                Reasoning as in~\cite{MR4234152,chen2024bayesiansamplingusinginteracting},
            }
            we deduce from~\cref{theorem:mfl} that Wasserstein-$p$ empirical chaos holds for the ensemble Langevin sampler~\eqref{eq:system_rewritten},
            in the sense of~\eqref{eq:wasserstein_p_chaos} below.
            Indeed, an application of the triangle inequality gives
            \[
                \expect \Bigl[W_p(\mu^J_t, \mfl_t)\Bigr] \leq
                \expect \Bigl[W_p(\mu^J_t, \empmfl_t)\Bigr] + \expect \Bigl[ W_p(\empmfl_t, \mfl_t) \Bigr].
            \]
            Under the assumptions of~\cref{theorem:mfl},
            the first term on the right-hand side tends to 0 in the limit $J \to \infty$.
            On the other hand,
            by~\cite[Theorem 1]{MR3383341},
            the second term decreases as $J^{-\alpha}$,
            for a constant~$\alpha \in (0, \frac{1}{2}]$ depending on~$p$, the dimension~$d$ and the number of bounded moments of $\mfl_t$.
            Therefore, it holds that
            \begin{equation}
                \label{eq:wasserstein_p_chaos}
                \forall t \in [0, T], \qquad
                \lim_{J \to \infty} \expect \Bigl[ W_p(\mu^J_t, \mfl_t) \Bigr]
                = 0.
            \end{equation}
            We refer to~\cite{MR3383341},
            see also~\cite[Lemma 4.2]{ReviewChaintronI},
            for the explicit value of~$\alpha$,
            which is in general strictly smaller than the Monte Carlo rate $\frac{1}{2}$,
            even when $\mfl_t$ has infinitely many moments.

        \item \revise{%
                In this proof, we aimed at obtaining an estimate that is explicit and quantitative in the parameter~$J$,
                the number of particles.
                However, we emphasize that the constant $C$ in~\eqref{eq:statement_main_theorem} also depends on the parameters $\l$ and $\u$ from~\cref{assumption:main},
                as well as on the dimension~$d$ and the final time~$T$.
                In particular, since we use Gr\"onwall's lemma in the proof,
                the dependence on~$T$ is at least exponential.
                We leave the precise investigation of the dependence of~$C$ on the parameters for future work,
                but note that determining this dependence is not simple,
                as the parameters appear not only in this proof but also in the auxiliary results.
                Notably, the constants~$\kappa$ and $\overline \kappa$ in the moment bounds~\eqref{eq:moment_bounds_2} and~\eqref{eq:moment_bound_mfl} depend on $(T, \u, \l, d)$.
            }
    \end{itemize}
\end{remark}

\subsection{Extension: sharp propagation of chaos for locally Lipschitz~\texorpdfstring{$\nabla \phi$}{setting}}
\label{sub:non_lip_thm}

In order to prove sharp propagation of chaos for this case,
we need the following additional auxiliary result.
\begin{lemma}
    [Convexity inequality]
    \label{lemma:convexity}
    If $\phi \in \mathcal A(\ell)$ for some $\ell \geq 0$,
    then there are $c_1 > 0$ and $c_2 > 0$ such that
    \begin{align*}
        \forall (x, y) \in \real^d \times \real^d, \qquad
        \Bigl\langle y - x, \nabla \phi(y) - \nabla \phi(x) \Bigr\rangle
        &\geq c_1 \left(1 + \lvert x \rvert^{\ell} + \lvert y \rvert^{\ell} \right) \lvert y - x \rvert^2
        - c_2 \lvert y - x \rvert^2.
    \end{align*}
\end{lemma}
\begin{proof}
    By the fundamental theorem,
    it holds that
    \begin{equation}
        \label{eq:convexity_first_eq}
        \nabla \phi(y) - \nabla \phi(x)
        = \int_{0}^{1} \hessian \phi\bigl(x + t(y-x)\bigr) \, (y - x) \, \d t.
    \end{equation}
    Assumption~\eqref{assump:hessian_bounds} implies that
    there is $c_2 > 0$ such that
    \[
        \forall x \in \real^d,
        \qquad \hessian \phi(x) \succcurlyeq c_{\ell} |x|^{\ell} \I_d - c_2 \I_d.
    \]
    Therefore,
    taking the inner product with $y - x$ on both sides of~\eqref{eq:convexity_first_eq},
    we obtain that
    \begin{align*}
        \Bigl\langle y - x, \nabla \phi(y) - \nabla \phi(x) \Bigr\rangle
            &\geq \l \int_{0}^{1} \left\lvert x + t(y-x) \right\rvert^{\ell} \, \lvert y - x \rvert^2 \, \d t
            - c_2 \int_{0}^{1} \lvert y - x \rvert^2 \, \d t.
    \end{align*}
    Suppose without loss of generality that $|y| \geq |x|$.
    Then, using the triangle inequality,
    we obtain that
    \begin{align*}
        \int_{0}^{1} \left\lvert x + t(y-x) \right\rvert^{\ell} \, \d t
        &\geq
        \int_{\frac{3}{4}}^{1} \left\lvert t y + (1 - t) x \right\rvert^{\ell} \, \d t
        \geq
        \int_{\frac{3}{4}}^{1} \Bigl( t |y| - (1-t) \lvert x \rvert \Bigr) ^{\ell} \, \d t \\
        &\geq
        \int_{\frac{3}{4}}^{1} \Bigl( 2t - 1 \Bigr) |y|^{\ell} \, \d t
        \geq C |y|^{\ell} \geq \frac{C}{2} \left(|x|^{\ell} + |y|^{\ell}\right),
    \end{align*}
    which concludes the proof.
\end{proof}

In the following result,
we assume for simplicity that the probability measure $\mfl_0$ has infinitely many moments,
but this assumption is not required.
As in the proof of~\cref{theorem:mfl},
the main idea of the proof is to introduce appropriate stopping times.
The key to handle the case where $\ell > 0$ is to define these stopping times in such a manner that the first term~$\mathcal I^1_t$ on the right-hand side of~\eqref{eq:four_terms_nonlip},
where the lack of global Lipschitz continuity of $\nabla \phi$ would be problematic in the classical approach,
can be controlled.
A very similar idea is used in the proof of uniqueness of the solution to the mean field dynamics~\eqref{eq:synchronous_coupling},
which is presented at the end of~\cref{sub:proof_wellposed}.

\begin{theorem}
    \label{theorem:mfl_nonlip}
    Suppose that~$\phi \in \mathcal A(\ell)$ for some $\ell \geq 0$,
    and consider the systems~\eqref{eq:system_rewritten} and \eqref{eq:coupling-system-mfl} with the coefficients given in~\eqref{eq:eks-drift-and-diff}.
    Assume that $\mfl_0 \in \mathcal P_{q}(\real^d)$ for all $q \in \nat$ and that $\cov(\mfl_0) \succ 0$.
    Then for all~$p > 0$,
    there is~$C > 0$ independent of~$J$ such that
    \begin{equation}
        \label{eq:statement_main_theorem_nonlip}
        \forall J\in\N^+, \qquad
        \forall j \in \range{1}{J}, \qquad
        \expect \left[ \sup_{t\in[0,T]} \abs[\Big]{\xn{j}_t-\xnl{j}_t}^p \right]
        \leq C J^{-\frac{p}{2}}.
    \end{equation}
\end{theorem}
\begin{proof}
    It suffices to prove the statement for $p \geq 4$ because,
    by Jensen's inequality, it holds for all $p \leq \mathfrak p$ that
    \[
        \expect \left[ \sup_{t\in[0,T]} \abs[\Big]{\xn{j}_t-\xnl{j}_t}^p \right]
        \leq
        \left(\expect \left[ \sup_{t\in[0,T]} \abs[\Big]{\xn{j}_t-\xnl{j}_t}^{\mathfrak p} \right] \right)^{\frac{p}{\mathfrak p}}.
    \]
    Fix $p \geq 4$,
    as well as  $\r = 2p$ and \revise{$q = 4p^2(\ell + 1)$}.
    Fix also $\varepsilon > 0$ to be determined later,
    and $R \in (0, \infty)$ such that
    \begin{equation}
        \label{eq:definition_R_nonlip}
        R^r >  \expect \left[ \sup_{t \in [0, T]} \left\lvert \xnl{j}_t \right\rvert^{\r} \right].
    \end{equation}
    Let~$\theta_J := \min\Bigl\{\mathfrak t_J, \overline \tau_J\Bigr\}$,
    where the stopping times $\mathfrak t_J$ and $\overline \tau_J$ are given by
    \begin{subequations}
    \begin{align}
        \label{eq:stop_1}
        \mathfrak t_J &= \inf \Bigl\{ t \geq 0 : W_{\r}{\left(\mu^J_t, \empmfl_t\right)} \geq \varepsilon \Bigr\} , \\
        \label{eq:stop_2}
        \overline \tau_J &= \inf \Bigl\{ t \geq 0 : W_{\r(\ell + 1)}{\left(\empmfl_t, \delta_0\right)} \geq R \Bigr\}.
    \end{align}
    \end{subequations}
    By the triangle inequality and the definition of $\theta_J$,
    it holds for all $t \in [0, T]$ that
    \[
        W_{\r} {\left(\mu^J_{t \wedge \theta_J}, \delta_0\right)} \leq
        W_{\r} {\left(\mu^J_{t \wedge \theta_J}, \empmfl_{t \wedge \theta_J}\right)} +
        W_{\r} {\left(\empmfl_{t \wedge \theta_J}, \delta_0\right)}
        \leq \varepsilon + R.
    \]
    In addition, in view of the Wasserstein stability estimate~\eqref{eq:stab_wcov_simple},
    it holds that
    \begin{equation}
        \label{eq:bound_covariances}
        \Bigl\lVert \cov(\mu^J_{t \wedge \theta_J}) - \cov(\empmfl_{t \wedge \theta_J}) \Bigr\rVert_{\rm F}
        \leq 2 (2R + \varepsilon) \varepsilon.
    \end{equation}
    Fix $j \in \range{1}{J}$.
    As in the proof of~\cref{theorem:mfl},
    we have by H\"older's inequality that
    \begin{align}
        &\expect \left[ \sup_{t\in[0,T]} \abs[\Big]{\xn{j}_t - \xnl{j}_t}^p \right]
        \label{eq:main_equation_nonlip}
        \leq \expect \left[ \sup_{t\in[0,T]} \abs[\Big]{\xn{j}_{t \wedge \theta_J} - \xnl{j}_{t \wedge \theta_J}}^p \right]
        + \left(\expect \left[ \sup_{t\in[0,T]} \abs[\Big]{\xn{j}_t - \xnl{j}_t}^{q} \right]\right)^{\frac{p}{q}}
        \Bigl( \proba [\theta_J \leq T]\Bigr)^{\frac{q-p}{q}}.
    \end{align}

    \paragraph{Step A. Bounding the first term in~\eqref{eq:main_equation_nonlip}}
    Here, the approach is slightly more technical than that in the proof of~\cref{theorem:mfl}.
    Recall that, by~\cref{proposition:well-posedness},
    there is $\revise{\overline \kappa} > 0$ such that
    \begin{equation}
        \label{eq:bounds_on_covariance}
        \forall t \in [0, T], \qquad
        \Bigl\lVert \cov (\mfl_t)  \Bigr\rVert_{\rm F}
        \vee
        \Bigl\lVert \cov (\mfl_t)^{-1}  \Bigr\rVert_{\rm F}
        \vee
        \left\lVert \frac{\d \cov (\mfl_t)}{\d t}  \right\rVert_{\rm F} \leq \revise{\overline \kappa}.
    \end{equation}
    Using this bound,
    applying It\^o's lemma to $f(x, \overline x, t) = \frac{1}{2} (x - \overline x)^\t \cov(\mfl_t)^{-1} (x - \overline x)$,
    and noting that $f\left(\xn{j}_0, \xnl{j}_0, 0\right) = 0$,
    we obtain that, for all $t \in [0, T]$,
    \begin{align}
        \notag
        \frac{1}{\revise{\overline \kappa}} \left\lvert \xn{j}_t - \xnl{j}_t \right\rvert^2
        \leq
        f\left(\xn{j}_t, \xnl{j}_t, t\right)
        &= - \int_{0}^{t} \Bigl\langle \xn{j}_s - \xnl{j}_s, \nabla \phi(\xn{j}_s) - \nabla \phi (\xnl{j}_s) \Bigr\rangle \, \d s \\
        \notag
        &\qquad + \int_{0}^{t} \Bigl\langle \xn{j}_s - \xnl{j}_s, \cov(\mfl_s)^{-1} \Bigl(\cov(\mfl_s) - \cov(\mu^J_s)\Bigr) \nabla \phi (\xn{j}_s) \Bigr\rangle \, \d s \\
        \notag
        &\qquad
        + \int_{0}^{t}
        \begin{pmatrix}
            \cov(\mu^J_s) & \sqrt{\cov(\mu^J_s)} \sqrt{\cov(\mfl_s)} \\
            \sqrt{\cov(\mfl_s)} \sqrt{\cov(\mu^J_s)} & \cov(\mfl_s)
        \end{pmatrix}
        :
        \begin{pmatrix}
            \cov(\mfl_s)^{-1} & - \cov(\mfl_s)^{-1} \\
            - \cov(\mfl_s)^{-1} & \cov(\mfl_s)^{-1}
        \end{pmatrix}
        \, \d s \\
        &\qquad
        \notag
        \revise{-} \int_{0}^{t} \frac{1}{2} \bigl(\xn{j}_s - \xnl{j}_s \bigr)^\t \cov(\mfl_s)^{-1} \frac{\d \cov(\mfl_s)}{\d s} \cov(\mfl_s)^{-1} \bigl( \xn{j}_s - \xnl{j}_s \bigr)^\t \, \d s \\
        &\qquad
        \label{eq:five_terms}
        + \int_{0}^{t}  \left(\xn{j}_s - \xnl{j}_s\right)^\t \cov(\mfl_s)^{-1} \left( \sqrt{2 \cov(\mu^J_s)} - \sqrt{2 \cov(\mfl_s)} \right) \, \d \wn{j}_s.
    \end{align}
    Let us bound the terms one by one.
    \begin{itemize}
        \item
            By~\cref{lemma:convexity} and~\cref{assumption:main},
            the integrand in the first term satisfies
            \begin{align}
                \notag
                - \Bigl\langle \xn{j}_s - \xnl{j}_s, \nabla \phi(\xn{j}_s) - \nabla \phi (\xnl{j}_s) \Bigr\rangle
                &\leq -c_1 \left(1 + \left\lvert \xn{j}_s \right\rvert^{\ell} + \left\lvert \xnl{j}_s \right\rvert^{\ell} \right)
                \left\lvert \xn{j}_s - \xnl{j}_s \right\rvert^{2} + c_2 \left\lvert \xn{j}_s - \xnl{j}_s \right\rvert^{2} \\
                \label{eq:good_term}
                &\leq - \frac{c_1}{\lipphi} \left\lvert \xn{j}_s - \xnl{j}_s \right\rvert \left\lvert \nabla \phi(\xn{j}_s) - \nabla \phi (\xnl{j}_s) \right\rvert
                + c_2 \left\lvert \xn{j}_s - \xnl{j}_s \right\rvert^{2}.
            \end{align}
            The first term on the right-hand side is a ``good'' term;
            its sign is negative.
            We shall use it to compensate a term coming from the second integral in~\eqref{eq:five_terms},
            which would otherwise be an obstacle in the proof.
            Propagation of chaos could still be proved without this compensation,
            i.e.\ if we simply omit in~\eqref{eq:four_terms_nonlip} the first term on the right-hand side of~\eqref{eq:good_term},
            but only a non-quantitative version.
            See~\cref{remark:nonquant} after the proof for more details on this point.

        \item
            By~\eqref{eq:bounds_on_covariance} and the triangle inequality,
            the integrand in the second term satisfies
            \begin{align}
                \notag
                &\Bigl\langle \xn{j}_s - \xnl{j}_s, \cov(\mfl_s)^{-1} \Bigl(\cov(\mfl_s) - \cov(\mu^J_s)\Bigr) \nabla \phi (\xn{j}_s) \Bigr\rangle \\
                \label{eq:trade1}
                &\qquad \leq
                \revise{\overline \kappa} \left\lvert \xn{j}_s - \xnl{j}_s \right\rvert
                \left\lVert \cov(\mfl_s) - \cov(\empmfl_s) \right\rVert_{\rm F}
                \left\lvert \nabla \phi (\xn{j}_s) \right\rvert
                + \revise{\overline \kappa} \left\lvert \xn{j}_s - \xnl{j}_s \right\rvert \left\lVert \cov(\empmfl_s) - \cov(\mu^J_s) \right\rVert_{\rm F} \left\lvert \nabla \phi (\xn{j}_s) \right\rvert.
            \end{align}
            \revise{%
                Using the inequality $\left\lvert \nabla \phi (\xn{j}_s) \right\rvert \leq \left\lvert \nabla \phi (\xnl{j}_s) \right\rvert + \left\lvert \nabla \phi(\xn{j}_s) - \nabla \phi (\xnl{j}_s) \right\rvert$,
                we have
                \begin{align}
                    \notag
                    &\Bigl\langle \xn{j}_s - \xnl{j}_s, \cov(\mfl_s)^{-1} \Bigl(\cov(\mfl_s) - \cov(\mu^J_s)\Bigr) \nabla \phi (\xn{j}_s) \Bigr\rangle \\
                    \notag
                    &\qquad \leq
                    \revise{\overline \kappa} \left\lvert \xn{j}_s - \xnl{j}_s \right\rvert
                    \left\lVert \cov(\mfl_s) - \cov(\empmfl_s) \right\rVert_{\rm F}
                    \left\lvert \nabla \phi (\xn{j}_s) \right\rvert
                    + \revise{\overline \kappa} \left\lvert \xn{j}_s - \xnl{j}_s \right\rvert \left\lVert \cov(\empmfl_s) - \cov(\mu^J_s) \right\rVert_{\rm F} \left\lvert \nabla \phi (\xnl{j}_s) \right\rvert \\
                    \label{eq:trade2}
                    & \qquad \qquad + \revise{\overline \kappa} \left\lvert \xn{j}_s - \xnl{j}_s \right\rvert
                \left\lVert \cov(\empmfl_s) - \cov(\mu^J_s) \right\rVert_{\rm F} \left\lvert \nabla \phi(\xn{j}_s) - \nabla \phi (\xnl{j}_s) \right\rvert.
                \end{align}
                Notice that, between~\eqref{eq:trade1} and~\eqref{eq:trade2},
                we traded the term $\left\lVert \cov(\empmfl_s) - \cov(\mu^J_s) \right\rVert_{\rm F}^2 \left\lvert \nabla \phi (\xn{j}_s) \right\rvert^2$
                for the nicer term~$\left\lVert \cov(\empmfl_s) - \cov(\mu^J_s) \right\rVert_{\rm F}^2 \bigl\lvert \nabla \phi (\xnl{j}_s) \bigr\rvert^2$.
                The latter term is simpler to handle in view of the definition of the stopping time~$\overline \tau_J$ in~\eqref{eq:stop_2}.
                The price of this trade is the presence of the last term on the right-hand side of~\eqref{eq:trade2} which,
                later in the proof, will be compensated with the first term in~\eqref{eq:good_term}.
                Using Young's inequality in the first two terms in~\eqref{eq:trade2}, we obtain
            }
            \begin{align*}
                &\Bigl\langle \xn{j}_s - \xnl{j}_s, \cov(\mfl_s)^{-1} \Bigl(\cov(\mfl_s) - \cov(\mu^J_s)\Bigr) \nabla \phi (\xn{j}_s) \Bigr\rangle \\
                &\qquad \leq
                \revise{\overline \kappa} \left\lvert \xn{j}_s - \xnl{j}_s \right\rvert^2
                + \frac{\revise{\overline \kappa}}{2} \left\lVert \cov(\mfl_s) - \cov(\empmfl_s) \right\rVert_{\rm F}^2
                  \left\lvert \nabla \phi (\xn{j}_s) \right\rvert^2
                + \frac{\revise{\overline \kappa}}{2}\left\lVert \cov(\empmfl_s) - \cov(\mu^J_s) \right\rVert_{\rm F}^2 \left\lvert \nabla \phi (\xnl{j}_s) \right\rvert^2 \\
                &\qquad \qquad + \revise{\overline \kappa} \left\lvert \xn{j}_s - \xnl{j}_s \right\rvert
                \left\lVert \cov(\empmfl_s) - \cov(\mu^J_s) \right\rVert_{\rm F} \left\lvert \nabla \phi(\xn{j}_s) - \nabla \phi (\xnl{j}_s) \right\rvert.
            \end{align*}

        \item
            For the third term in~\eqref{eq:five_terms},
            a simple calculation gives that
            \begin{subequations}
                \begin{align}
                    \notag
                &\begin{pmatrix}
                    \cov(\mu^J_s) & \sqrt{\cov(\mu^J_s)} \sqrt{\cov(\mfl_s)} \\
                    \sqrt{\cov(\mfl_s)} \sqrt{\cov(\mu^J_s)} & \cov(\mfl_s)
                \end{pmatrix}
                :
                \begin{pmatrix}
                    \cov(\mfl_s)^{-1} & - \cov(\mfl_s)^{-1} \\
                    - \cov(\mfl_s)^{-1} & \cov(\mfl_s)^{-1}
                \end{pmatrix} \\
                \notag
                &\qquad \qquad
                =
                \trace \left( \Bigl(\sqrt{\cov(\mu^J_s)} - \sqrt{\cov(\mfl_s)}\Bigr) \Bigl(\sqrt{\cov(\mu^J_s)} - \sqrt{\cov(\mfl_s)}\Bigr) \cov(\mfl_s)^{-1} \right) \\
                \label{eq:second_to_last}
                &\qquad \qquad
                = \left\lVert \sqrt{\cov(\mfl_s)^{-1}}\Bigl(\sqrt{\cov(\mu^J_s)} - \sqrt{\cov(\mfl_s)}\Bigr) \right\rVert_{\rm F}^2
                \leq \revise{\overline \kappa} \left\lVert \sqrt{\cov(\mu^J_s)} - \sqrt{\cov(\mfl_s)} \right\rVert_{\rm F}^2 \\
                \label{eq:last_ineq}
                &\qquad \qquad
                \leq 2 \revise{\overline \kappa} W_2\left(\mu^J_s, \empmfl_s\right)^2 + 2 \revise{\overline \kappa} \left\lVert \sqrt{\cov(\empmfl_s)} - \sqrt{\cov(\mfl_s)} \right\rVert_{\rm F}^2.
                \end{align}
            \end{subequations}
            where we used~\eqref{eq:bounds_on_covariance} in~\eqref{eq:second_to_last},
            and then the triangle inequality and
            the Wasserstein stability estimate~\eqref{eq:wmean-wcov-emp-local-lip} from~\cref{lemma:wasserstein_stability_estimates}
            in~\eqref{eq:last_ineq}.

        \item
            For the fourth term in~\eqref{eq:five_terms},
            we have in view of the bounds~\eqref{eq:bounds_on_covariance} on the mean field covariance
            that
            \[
                \bigl(\xn{j}_s - \xnl{j}_s \bigr)^\t \cov(\mfl_s)^{-1} \frac{\d \cov(\mfl_s)}{\d s} \cov(\mfl_s)^{-1} \bigl( \xn{j}_s - \xnl{j}_s \bigr)^\t
                \leq \revise{\overline \kappa}^3 \left\lvert \xn{j}_s - \xnl{j}_s \right\rvert^2.
            \]

        \item
            Finally, for the last term in~\eqref{eq:five_terms},
            let $M_t$ denote the martingale
            \[
                M^j_t = \int_{0}^{t}  \left(\xn{j}_s - \xnl{j}_s\right)^\t \cov(\mfl_t)^{-1} \left( \sqrt{2 \cov(\mu^J_s)} - \sqrt{2 \cov(\mfl_s)} \right) \, \d \wn{j}_s.
            \]
            By It\^o's isometry, the bounds ~\eqref{eq:bounds_on_covariance},
            the triangle inequality,
            and the Wasserstein stability estimate~\eqref{eq:wmean-wcov-emp-local-lip},
            the quadratic variation of this process is bounded from above as follows:
            \begin{align}
                \notag
                \langle M^j \rangle_t
                &\leq C \int_{0}^{t}  \left\lvert \xn{j}_s - \xnl{j}_s \right\rvert^2 \left\lVert \sqrt{\cov(\mu^J_s)} - \sqrt{\cov(\mfl_s)} \right\rVert_{\rm F}^2 \, \d s \\
                \label{eq:quad_var}
                &\leq C \int_{0}^{t}  \left\lvert \xn{j}_s - \xnl{j}_s \right\rvert^4 + W_2{\left(\mu^J_s, \empmfl_s\right)}^4 + \left\lVert \sqrt{\cov(\empmfl_s)} - \sqrt{\cov(\mfl_s)} \right\rVert_{\rm F}^4 \, \d s.
            \end{align}
    \end{itemize}
    Combining all the bounds,
    we obtain
    \begin{align}
        \notag
        \frac{1}{\revise{\overline \kappa}}
        \left\lvert \xn{j}_t - \xnl{j}_t \right\rvert^2
        &\leq
        \int_{0}^{t}
        \left( \revise{\overline \kappa} \Bigl\lVert \cov(\mu^J_s) - \cov(\empmfl_s) \Bigr\rVert_{\rm F}- \frac{c_1}{\lipphi} \right) \left\lvert \xn{j}_s - \xnl{j}_s \right\rvert \left\lvert \nabla \phi(\xn{j}_s) - \nabla \phi (\xnl{j}_s) \right\rvert \, \d s
        \\ \notag
        &\qquad +  C \int_{0}^{t}  \left\lvert \xn{j}_s - \xnl{j}_s \right\rvert^2 + W_2{\left(\mu^J_s, \empmfl_s\right)}^2
        +  \Bigl\lVert \cov(\mu^J_s) - \cov(\empmfl_s) \Bigr\rVert_{\rm F}^2 \left\lvert \nabla \phi (\xnl{j}_s) \right\rvert^2 \, \d s
        \\ \notag
        &\qquad + C \int_{0}^{t}  \Bigl\lVert \cov(\empmfl_s) - \cov(\mfl_s) \Bigr\rVert_{\rm F}^2
                \left\lvert \nabla \phi (\xn{j}_s) \right\rvert^2
        + \left\lVert \sqrt{\cov(\empmfl_s)} - \sqrt{\cov(\mfl_s)} \right\rVert_{\rm F}^2 \, \d s
        + M^j_t \\
        \label{eq:four_terms_nonlip}
        &=: \mathcal I^1_t + \mathcal I^2_t + \mathcal I^3_t + M^j_t.
    \end{align}
    In view of~\eqref{eq:bound_covariances},
    it holds that
    \[
        \mathcal I^1_{u \wedge \theta_J}
        \leq \left( 2 \revise{\overline{\kappa}} (2R + \varepsilon) \varepsilon - \frac{c_1}{\lipphi} \right) \int_{0}^{u \wedge \theta_J} \left\lvert \xn{j}_s - \xnl{j}_s \right\rvert \left\lvert \nabla \phi(\xn{j}_s) - \nabla \phi (\xnl{j}_s) \right\rvert \, \d s.
    \]
    Let $\varepsilon$ be such that the first factor on the right-hand side is negative.
    Then,
    evaluating both sides of~\eqref{eq:four_terms_nonlip} at the stopping time~$u \wedge \theta_J$,
    raising to the power $\frac{\r}{2}$,
    and taking the supremum \revise{over $[0, t]$},
    we obtain
    \begin{align}
        \notag
        \frac{1}{3^{\frac{\r}{2} - 1}}\sup_{u \in [0, t]} \left\lvert \xn{j}_{u \wedge \theta_J} - \xnl{j}_{u \wedge \theta_J} \right\rvert^{\r}
        &\leq \sup_{u \in [0, t]} \left\lvert \mathcal I^2_{u \wedge \theta_J} \right\rvert^{\frac{\r}{2}}
        + \sup_{u \in [0, t]} \left\lvert \mathcal I^3_{u \wedge \theta_J} \right\rvert^{\frac{\r}{2}}
        + \sup_{u \in [0, t]} \left\lvert M^j_{\revise{u} \wedge \theta_J} \right\rvert^{\frac{\r}{2}} \\
        \label{eq:second_ineq}
        &\leq \left\lvert \mathcal I^2_{t \wedge \theta_J} \right\rvert^{\frac{\r}{2}}
        + \left\lvert \mathcal I^3_{t} \right\rvert^{\frac{\r}{2}}
        + \sup_{u \in [0, t]} \left\lvert M^j_{\revise{u} \wedge \theta_J} \right\rvert^{\frac{\r}{2}}.
    \end{align}
    \revise{We now use that}
    \[
        \left\lvert \mathcal I^2_{t \wedge \theta_J} \right\rvert^{\frac{\r}{2}}
        \leq
        C \int_{0}^{t}  \left\lvert \xn{j}_{s\wedge \theta_J} - \xnl{j}_{s\wedge \theta_J} \right\rvert^{\r} + W_2{\left(\mu^J_{s\wedge \theta_J}, \empmfl_{s\wedge \theta_J}\right)}^{\r}
        +  \Bigl\lVert \cov(\mu^J_{s\wedge \theta_J}) - \cov(\empmfl_{s\wedge \theta_J}) \Bigr\rVert_{\rm F}^{\r} \left\lvert \nabla \phi (\xnl{j}_{s\wedge \theta_J}) \right\rvert^{\r} \, \d s.
    \]
    Taking the expectation in~\eqref{eq:second_ineq},
    then using~\eqref{eq:bound_wasserstein},
    the Burkholder--Davis--Gundy inequality and the optional stopping theorem,
    and finally using the bound~\eqref{eq:quad_var} on the quadratic variation of $M^j$,
    we deduce that
    \begin{align*}
        \expect \left[ \sup_{u \in [0, t]} \left\lvert \xn{j}_{u \wedge \theta_J} - \xnl{j}_{u \wedge \theta_J} \right\rvert^{\r} \right]
        &\leq
        C \int_{0}^{t}  \expect \left\lvert \xn{j}_{s\wedge \theta_J} - \xnl{j}_{s\wedge \theta_J} \right\rvert^{\r}
         + \expect \biggl[ \Bigl\lVert \cov(\mu^J_{s\wedge \theta_J}) - \cov(\empmfl_{s\wedge \theta_J}) \Bigr\rVert_{\rm F}^{\r}
         \left\lvert \nabla \phi (\xnl{j}_{s\wedge \theta_J}) \right\rvert^{\r} \biggr] \, \d s \\
        &\qquad + C \int_{0}^{t}
        \expect \biggl[\Bigl\lVert \cov(\empmfl_s) - \cov(\mfl_s) \Bigr\rVert_{\rm F}^{\r}
            \left\lvert \nabla \phi (\xn{j}_s) \right\rvert^{\r} \biggr]
                + \expect \left\lVert \sqrt{\cov(\empmfl_s)} - \sqrt{\cov(\mfl_t)} \right\rVert_{\rm F}^{2\r} \, \d s.
    \end{align*}
    By exchangeability,
    the Wasserstein stability estimate~\eqref{eq:stab_wcov_simple} in~\cref{lemma:wasserstein_stability_estimates},
    and the definition of the stopping time~$\theta_J$,
    it holds for some $C = C(R, \varepsilon)$ that
    \begin{align}
        \notag
        \expect \biggl[ \Bigl\lVert \cov(\mu^J_{s \wedge \theta_J}) - \cov(\empmfl_{s \wedge \theta_J}) \Bigr\rVert_{\rm F}^{\r} \, \left\lvert \nabla \phi (\xnl{j}_{s \wedge \theta_J}) \right\rvert^{\r} \biggr]
        &= \expect \left[ \Bigl\lVert \cov(\mu^J_{s \wedge \theta_J}) - \cov(\empmfl_{s \wedge \theta_J}) \Bigr\rVert_{\rm F}^{\r} \, \frac{1}{J} \sum_{k=1}^{J} \left\lvert \nabla \phi (\xnl{k}_{s \wedge \theta_J}) \right\rvert^{\r} \right] \\
        \label{eq:difficult_term}
        &\leq C \expect \left[ W_2{\left(\mu^J_{s \wedge \theta_J}, \empmfl_{s \wedge \theta_J}\right)}^{\r} \, \frac{1}{J} \sum_{k=1}^{J} \left\lvert \xnl{k}_{s \wedge \theta_J} \right\rvert_*^{\r(\ell + 1)} \right] \\
        \notag
        &\leq C \expect \left[ W_2{\left(\mu^J_{s \wedge \theta_J}, \empmfl_{s \wedge \theta_J}\right)}^{\r} \right]
        \leq C \expect \left\lvert \xn{j}_{s \wedge \theta_J} - \xnl{j}_{s \wedge \theta_J} \right\rvert^{\r}.
    \end{align}
    Note that, \revise{to exploit the definition of the stopping time $\mu_J$},
    it was important to have $\xnl{k}$ in the sum in~\eqref{eq:difficult_term}, and not $\xn{k}$.
    On the other hand,
    by the moment bounds in~\cref{proposition:moment_estimates,proposition:well-posedness} and by~\cref{lemma:convergence_covariance_iid},
    it holds that
    \begin{align*}
        \expect \left\lVert \sqrt{\cov(\empmfl_s)} - \sqrt{\cov(\mfl_s)} \right\rVert_{\rm F}^{\r}
        &\leq C J^{-\frac{\r}{2}}, \\
        \expect \left[ \left\lVert \cov(\empmfl_s) - \cov(\mfl_s) \right\rVert_{\rm F}^{\r} \left\lvert \nabla \phi (\xn{j}_s) \right\rvert^\r \right]
        &\leq
        \left( \expect  \left\lVert \cov(\empmfl_s) - \cov(\mfl_s) \right\rVert_{\rm F}^{2\r}  \right)^{\frac{1}{2}} \left( \expect \left\lvert \nabla \phi (\xn{j}_s) \right\rvert^{2\r} \right)^{\frac{1}{2}}
        \leq C J^{-\frac{\r}{2}}.
    \end{align*}
    Therefore, we obtain that
    \begin{align*}
        \expect \left[ \sup_{u \in [0, t]} \left\lvert \xn{j}_{u \wedge \theta_J} - \xnl{j}_{u \wedge \theta_J} \right\rvert^{\r} \right]
        &\leq  C J^{-\frac{\r}{2}} + C \int_{0}^{t}  \expect \left\lvert \xn{j}_{s \wedge \theta_J} - \xnl{j}_{s \wedge \theta_J} \right\rvert^{\r} \, \d s.
    \end{align*}
    By Gr\"onwall's inequality,
    this implies that
    \begin{equation}
        \label{eq:gronwall_estimate}
        \expect \left[ \sup_{u \in [0, T]} \left\lvert \xn{j}_{u \wedge \theta_J} - \xnl{j}_{u \wedge \theta_J} \right\rvert^{\r} \right]
        \leq C J^{-\frac{\r}{2}}.
    \end{equation}

    \paragraph{Step B. Bounding the second term in~\eqref{eq:main_equation_nonlip}}
    We have by~\cref{proposition:well-posedness,proposition:moment_estimates} that
    \[
        \expect \left[ \sup_{t\in[0,T]} \abs[\big]{\xn{j}_t - \xnl{j}_t}^{q} \right]
        \leq 2^{q-1} \left( \expect \left[ \sup_{t\in[0,T]} \abs[\big]{\xn{j}_t}^{q} \right]
        + \expect \left[ \sup_{t\in[0,T]} \abs[\big]{\xnl{j}_t}^{q} \right] \right)
        \leq 2^{q-1} \bigl( \kappa(q) + \overline \kappa(q) \bigr).
    \]
    In order to complete the proof of the theorem,
    it remains to bound the probability $\proba \left[\theta_J(R) \leq T \right]$.
    By using the same strategy as in~\eqref{eq:bound_second_big_term} in the proof of~\cref{theorem:mfl},
    it is simple to show that
    there is $C$ such that
    \[
        \proba \left[ \overline \tau_J \leq T \right]
        \leq C J^{-\frac{q}{2 r(\ell + 1)}}.
    \]
    On the other hand, by the Markov inequality and~\eqref{eq:gronwall_estimate},
    it holds that
    \begin{align*}
        \proba \left[ \mathfrak t_J \leq T \leq \overline \tau_J \right]
        &= \proba \left[ \sup_{t \in [0, T]} W_{\r}{\left(\mu^J_{t \wedge \theta_J}, \empmfl_{t \wedge \theta_J}\right)} \geq \varepsilon \right]
        \leq \proba \left[ \sup_{t \in [0, T]}  \frac{1}{J} \sum_{j=1}^{J} \left\lvert \xn{j}_{t \wedge \theta_J} - \xnl{j}_{t \wedge \theta_J} \right\rvert^{\r}  \geq \varepsilon^r \right] \\
        &\leq \proba \left[  \frac{1}{J} \sum_{j=1}^{J} \sup_{t \in [0, T]}  \left\lvert \xn{j}_{t \wedge \theta_J} - \xnl{j}_{t \wedge \theta_J} \right\rvert^{\r}  \geq \varepsilon^r \right]
        \leq \frac{1}{\varepsilon^r} \expect \left[ \sup_{t \in [0, T]}  \left\lvert \xn{j}_{t \wedge \theta_J} - \xnl{j}_{t \wedge \theta_J} \right\rvert^{\r} \right]
        \leq \frac{C}{\varepsilon^r} J^{-{\frac{r}{2}}}.
    \end{align*}
    Given the definitions of $p$ and $q$,
    it follows that
    \[
        \Bigl( \proba [\theta_J \leq T]\Bigr)^{\frac{q-p}{q}}
        \leq C \Bigl(J^{-\frac{q}{2 r(\ell + 1)}} + J^{-\frac{\r}{2}}\Bigr)^{\frac{q-p}{q}}
        = C \Bigl(J^{-p} + J^{-p}\Bigr)^{1 - \frac{1}{4p(\ell + 1)}}
        \leq C J^{-\frac{p}{2}}.
    \]
    The proof can then be concluded in exactly the same manner as the proof of~\cref{theorem:mfl}.
\end{proof}
\begin{remark}
    [Non-quantitative propagation of chaos]
    \label{remark:nonquant}
    A natural alternative approach would have been to define stopping times
    similar to those employed in the proof of~\cref{theorem:mfl},
    \begin{equation}
        \label{eq:non_quant}
        \tau_J(r, R) = \inf \Bigl\{ t \geq 0 : W_{\r(\ell + 1)}(\emp_t, \delta_0) \geq R \Bigr\},
        \qquad
        \overline \tau_J(r, R) = \inf \Bigl\{ t \geq 0 : W_{\r(\ell + 1)}(\empmfl_t, \delta_0) \geq R \Bigr\},
    \end{equation}
    and to let $\theta_J(r, R) = \tau_J(r, R) \wedge \overline \tau_J(r, R)$.
    This remark aims at explaining why,
    although this approach does enable proving a non-quantitative propagation of chaos estimate,
    it does not yield a \revise{good} quantitative result.

    With the definition~\eqref{eq:non_quant} of the stopping times,
    the compensation that we relied on to control the first term in~\eqref{eq:four_terms_nonlip} can no longer be applied.
    This leads to the presence,
    among the terms needing to be controlled,
    of a term similar to~\eqref{eq:difficult_term} but containing $\xn{k}_{s \wedge \theta_J}$ instead of $\xnl{k}_{s \wedge \theta_J}$.
    This is not a problem,
    as this term can be handled by definition of~$\theta_J$ in terms of the stopping times in~\eqref{eq:non_quant}.
    Consequently,
    the estimate~\eqref{eq:gronwall_estimate} still holds,
    with a constant~$C$ depending on~$(r,R)$.
    The issue with this approach is to obtain a sufficiently good bound on~$\proba [\tau_J(R) \leq T]$ in the last part of the proof.
    Indeed, while it is still possible to show,
    similarly to~\eqref{eq:auxiliary_bound_proba},
    that
    \[
        \proba [\tau_J \leq T]
        \leq
        \frac{2^{\beta}}{R^{\beta}} \expect \left[ \frac{1}{J}\sum_{j=1}^J \sup_{t\in[0,T]} \Bigl\lvert \xn{j}_{t\wedge \theta_J} - \xnl{j}_{t\wedge \theta_J} \Bigr\rvert^{\beta} \right]
        + \proba \left[\frac{1}{J}\sum_{j=1}^J \sup_{t\in[0,T]} \Bigl\lvert \xnl{j}_t \Bigr\rvert^{\beta} \geq \frac{R^{\beta}}{2^{\beta}} \right], \qquad \beta = r(\ell + 1),
    \]
    the first term on the right-hand side of this inequality cannot be controlled from the estimate~\eqref{eq:gronwall_estimate} on the stopped particle system,
    unless $\beta = r$, which is the globally Lipschitz setting.
    A crude alternative manner of bounding $\proba [\tau_J \leq T]$ is to use the Markov inequality,
    which together with the moment bound in~\cref{proposition:moment_estimates} gives that
    \begin{align*}
        \proba \left[ \tau_J(r, R) \leq T \right]
        = \proba \left[ \sup_{t \in [0, T]} \frac{1}{J} \sum_{j=1}^{J} \left\lvert \xn{j}_t \right\rvert^\beta \geq R^\beta \right]
        &\leq \proba \left[ \frac{1}{J} \sum_{j=1}^{J} \sup_{t \in [0, T]} \left\lvert \xn{j}_t \right\rvert^\beta \geq R^\beta \right] \\
        &\leq \frac{1}{R^\beta}\expect \left[ \frac{1}{J} \sum_{j=1}^{J} \sup_{t \in [0, T]} \left\lvert \xn{j}_t \right\rvert^\beta \right]
        = \frac{1}{R^\beta}\expect \left[ \sup_{t \in [0, T]} \left\lvert \xn{1}_t \right\rvert^\beta \right] \leq \frac{\kappa(\beta)}{R^\beta}.
    \end{align*}
    A non-quantitative version of propagation can then be derived as follows.
    Take $\varepsilon > 0$ and
    fix $R > 0$ so that the second term in~\eqref{eq:main_equation_nonlip},
    with the stopping times given in~\eqref{eq:non_quant},
    is bounded from above by~$\frac{\varepsilon}{2}$.
    This is possible in view of the crude bound on $\proba \left[ \tau_J(r, R) \leq T \right]$ above,
    together with a similar bound for $\proba \left[ \overline \tau_J(r, R) \leq T \right]$.
    Then fix~$J$ sufficiently large such that the first term in~\eqref{eq:main_equation_nonlip} is also bounded from above by~$\frac{\varepsilon}{2}$,
    which is possible by~\eqref{eq:gronwall_estimate}.
    Since $\varepsilon$ was arbitrary,
    it follows that
    \[
        \lim_{J \to \infty}
        \expect \left[ \sup_{t\in[0,T]} \abs[\Big]{\xn{j}_t-\xnl{j}_t}^r \right]
        = 0,
    \]
    which is indeed a non-quantitative propagation of chaos estimate.
    Note that this approach of obtaining a non-quantitative estimate
    parallels that in~\cite[Theorem 2.2]{MR1949404}.
\end{remark}

\subsection{Corollary: bound on the sampling error}
\label{sub:corollary}
To conclude this section,
we mention the following corollary of~\cref{theorem:mfl,theorem:mfl_nonlip},
which is similar to~\cite[Theorem 2]{MR4199469} and useful in the context of sampling.
It states that, for the purpose of calculating the average of an observable~$f\colon \real^d \to \real$
with respect to~$\mfl_t$,
the interacting particle system at time~$t$ is as good,
in terms of convergence rate of \revise{the~$L^p$} approximation error with respect to~$J$,
as an estimator constructed from $J$~i.i.d.\ samples from~$\mfl_t$.

\begin{corollary}
    [$L^p$ bound on the sampling error]
    Suppose that $\phi \in \mathcal A(\ell)$ for some $\ell > 0$,
    that $\mfl_0 \in \mathcal P_q(\real^d)$ for all~$q > 0$ and that $\cov(\mfl_0) \succ 0$.
    Assume additionally that~$f \colon \real^d \to \real$ satisfies the following local Lipschitz continuity assumption:
    \begin{align}
        \label{eq:lip_corollary}
        \forall x, y \in \R^d, \qquad
        \bigl\lvert f(x)- f(y) \bigr\rvert
        \le L \Bigl( 1 + |x|^{s} + |y|^{s} \Bigr) |x-y|.
    \end{align}
    Consider the systems~\eqref{eq:system_rewritten} and \eqref{eq:coupling-system-mfl} with the coefficients given in~\eqref{eq:eks-drift-and-diff}.
    Then for any $p \geq \revise{2}$,
    there is $C > 0$ depending on $(p, L, s)$ and a finite number of moments of $\mfl_0$ such that
    \[
        \left( \expect  \left\lvert \frac{1}{J} \sum_{j=1}^{J} f(\xn{j}_t) - \mfl_t[f] \right\rvert^p \right)^{\frac{1}{p}}
        \leq
        C J^{-\frac{1}{2}}.
    \]
\end{corollary}

\begin{proof}
    By the triangle inequality,
    it holds that
    \begin{equation}
        \label{eq:sampling_error}
        \left( \expect  \left\lvert \frac{1}{J} \sum_{j=1}^{J} f(\xn{j}_t) - \mfl_t[f] \right\rvert^p \right)^{\frac{1}{p}}
        \leq
        \left( \expect  \left\lvert \frac{1}{J} \sum_{j=1}^{J} \Bigl(f\bigl(\xn{j}_t\bigr) - f\bigl(\xnl{j}_t\bigr) \Bigr) \right\rvert^p \right)^{\frac{1}{p}}
        + \left( \expect  \left\lvert \frac{1}{J} \sum_{j=1}^{J} f(\xnl{j}_t) - \mfl_t[f] \right\rvert^p \right)^{\frac{1}{p}}.
    \end{equation}
    By Jensen's inequality,
    exchangeability,
    the local Lipschitz continuity of $f$,
    the Cauchy--Schwarz inequality,
    the \revise{moment bounds~\eqref{eq:moment_bounds_2} and~\eqref{eq:moment_bound_mfl}},
    and~\cref{theorem:mfl_nonlip},
    the first term satisfies
    \begin{align*}
        \left( \expect  \left\lvert \frac{1}{J} \sum_{j=1}^{J} \Bigl(f\bigl(\xn{j}_t\bigr) - f\bigl(\xnl{j}_t\bigr) \Bigr) \right\rvert^p \right)^{\frac{1}{p}}
        &\leq \left( \expect \left[  \frac{1}{J} \sum_{j=1}^{J} \Bigl\lvert f\bigl(\xn{j}_t\bigr) - f\bigl(\xnl{j}_t\bigr) \Bigr\rvert^p  \right]     \right)^{\frac{1}{p}}
        = \left( \expect  \left\lvert f\bigl(\xn{1}_t\bigr) - f\bigl(\xnl{1}_t\bigr)  \right\rvert^p \right)^{\frac{1}{p}} \\
        &\leq
        L \left(3^{\revise{2p-1}}\expect\left[1 + \left\lvert \xn{1}_t \right\rvert^{2ps} + \bigl\lvert \xnl{1}_t \bigr\rvert^{2ps} \right] \right)^{\frac{1}{2p}}
        \left( \expect  \left\lvert \xn{1}_t - \xnl{1}_t  \right\rvert^{2p} \right)^{\frac{1}{2p}}
        \leq C J^{- \frac{1}{2}}.
    \end{align*}
    By the Marcinkiewicz--Zygmund inequality and the moment bound in~\cref{proposition:well-posedness},
    the second term on the right-hand side of~\eqref{eq:sampling_error} also tends to 0 at the classical Monte Carlo rate~$J^{-\frac{1}{2}}$,
    which concludes the proof.
\end{proof}

\appendix

\section{Proof of well-posedness and moment bounds}

We present first the proof of well-posedness for the interacting particle system in~\cref{sub:proof_moment_bounds},
then auxiliary results in~\cref{sub:aux_wellposed},
and finally the proof of well-posedness for the mean field dynamics,
relying on these auxiliary results, in~\cref{sub:proof_wellposed}.
In several proofs in this section,
we \revise{use} that if $\phi \in \mathcal A(\ell)$ and $f(x) = \phi(x)^q$ for $q > 0$,
then by the upper bounds in~\cref{assumption:main} (and the lower bound on $\phi$ if $0 < q < \revise{2}$),
it holds that
\begin{align}
    \nonumber
    \forall x \in \real^d, \qquad
    \bigl\lVert \hessian f(x) \bigr\rVert_{\rm F}
    &=  \Bigl\lVert q(q-1)  \phi(x)^{q-2} \nabla \phi(x) \otimes \nabla \phi(x) + q  \phi(x)^{q-1} \hessian \phi(x)   \Bigr\rVert_{\rm F}   \\
    \label{eq:aux_hessian}
    &\leq C \lvert x \rvert_*^{(q-2)(\ell+2) + 2(\ell +1)} + C \lvert x \rvert_*^{(q-1)(\ell+2) + \ell}
    = 2 C \lvert x \rvert_*^{q(\ell+2) - 2}.
\end{align}
Recall that $\phi$ is bounded from below by 1 by~\cref{assumption:main},
so the function $f$ is indeed well-defined and twice differentiable.

\label{sec:well-posed}

\subsection{\texorpdfstring{Proof of~\cref{proposition:moment_estimates}}{Proof of auxiliary lemma on moment estimates}}
\label{sub:proof_moment_bounds}
We first prove well-posedness,
and then the moment bound~\eqref{eq:moment_bounds_2}.

\paragraph{Step 1. Proving well-posedness of the interacting particle system}
This part is similar to the proof of~\cite[Proposition 4.4]{MR4123680},
but slightly simpler because we do not need to prove nondegeneracy of the empirical covariance.
Let $\mathcal L$ denote the generator of the interacting particle system
and let $\X \in \real^{dJ}$ be the collection $(\xn{1}, \dotsc, \xn{J})$.
Fix~$j \in \range{1}{J}$ and \revise{$q = \frac{p}{\ell + 2}$} and let~$\mathscr V$ denote the Lyapunov functional
\[
    \mathscr V(\X) = \frac{1}{J}\sum_{j=1}^{J} \phi(\xn{j})^q.
\]
Let also $f(x) = \phi(x)^q$ and $\mu^J = \frac{1}{J} \sum_{j=1}^{J} \delta_{\xn{j}}$.
It holds that
\begin{align*}
    \mathcal L \mathscr V(\X)
    &= \frac{1}{J} \sum_{j=1}^{J} \Bigl( - \nabla \phi(\xn{j})^\t \cov(\mu^J)  \nabla_{j} \mathscr V(\X)
    + \cov(\mu^J) : \hessian_j \mathscr V(\X) \Bigr) \\
    &=  \frac{1}{J}\sum_{j=1}^{J} \Bigl( - q \phi(\xn{j})^{q-1} \nabla \phi(\xn{j})^\t \cov(\mu^J)  \nabla \phi(\xn{j})
    + \cov(\mu^J) : \hessian f(\xn{j}) \Bigr)
    \leq \frac{1}{J} \sum_{j=1}^{J} \cov(\mu^J) : \hessian f(\xn{j}),
\end{align*}
where $\nabla_j$ denotes the gradient with respect to $\xn{j}$.
Using~\eqref{eq:aux_hessian},
we deduce that
\begin{align*}
    \mathcal L \mathscr V(\X)
    &\leq C \Bigl\lVert \cov(\mu^J) \Bigr\rVert_{\rm F} \biggl( \frac{1}{J} \sum_{j=1}^{J} \bigl\lvert \xn{j} \bigr\rvert_*^{q(\ell + 2) - 2} \biggr)
    \leq C \biggl(\frac{1}{J} \sum_{k=1}^{J} \bigl\lvert \xn{j} \bigr\rvert_*^{2}\biggr) \biggl(\frac{1}{J} \sum_{j=1}^{J} \bigl\lvert \xn{j} \bigr\rvert_*^{q(\ell + 2) - 2} \biggr) \\
    &= C \left(\int_{\real^d} |x|_*^2 \, \mu^J(\d x) \right) \left(\int_{\real^d} \bigl\lvert x \bigr\rvert_*^{q(\ell + 2) - 2} \, \mu^J(\d x) \right)
    \leq C \int_{\real^d} \bigl\lvert x \bigr\rvert_*^{q(\ell + 2)} \, \mu^J(\d x),
\end{align*}
where in the last \revise{bound, we used Jensen's inequality:
\[
    \left(\int_{\real^d} |x|_*^2 \, \mu^J(\d x) \right) \left(\int_{\real^d} \bigl\lvert x \bigr\rvert_*^{q(\ell + 2) - 2} \, \mu^J(\d x) \right)
    \leq \left(\int_{\real^d} |x|_*^{q(\ell + 2)} \, \mu^J(\d x) \right)^{\frac{2}{q(\ell + 2)}}
    \left(\int_{\real^d} |x|_*^{q(\ell + 2)} \, \mu^J(\d x) \right)^{\frac{q(\ell + 2) - 2}{q(\ell + 2)}}.
\]
}
Using~\eqref{eq:assumption_allx},
we conclude that
\begin{equation*}
    \label{eq:lyapunov_generator}
    \forall \X \in \real^{dJ}, \qquad
    \mathcal L \mathscr V(\X)
    \leq C \mathscr V(\X).
\end{equation*}
From this \revise{inequality},
it then follows from~\cite[Theorem 3.5]{khasminskii2013stochastic},
see also~\cite[Theorem 2.1]{MR1234295},
that there exists a unique strong globally-defined solution to the interacting particle system if $\mfl_0 \in \mathcal P_p(\real^d)$,
and that furthermore
\begin{equation}
    \label{eq:khasminskii_bound}
    \sup_{t \in [0, T]} \expect \bigl[\mathscr V(\X_t)\bigr]
    \leq \revise{\exp \left( C T\right) \expect \bigl[\mathscr V(\X_0)\bigr]}
    \leq \revise{\exp \left( C T\right) \u^q \expect \left\lvert \xn{1}_0 \right\rvert_*^{p}}
    < \infty.
\end{equation}

\paragraph{Step 2. Proving the moment bound~\eqref{eq:moment_bounds_2}}
Fix $j \in \range{1}{J}$ and let $f(x) = \phi(x)^q$,
\revise{for $q = \frac{p}{2(\ell + 2)}$}.
By It\^o's formula and a reasoning similar to above,
it holds that
\[
    f(\xn{j}_t)
    \leq f(\xn{j}_0) + \int_{0}^{t} \cov(\mu^J_s) : \hessian f(\xn{j}_s) \, \d s
    + \int_{0}^{t} \sqrt{2 \cov(\mu^J_s)} \nabla f(\xn{j}_s) \, \d \wn{j}_s.
\]
Taking the square and the supremum,
then taking the expectation and using the Burkholder--Davis--Gundy inequality,
we obtain that
\begin{align}
    \notag
    \frac{1}{3} \expect \left[ \sup_{s \in [0, t]} \bigl\lvert f(\xn{j}_s) \bigr\rvert^2 \right]
    &\leq \expect  \bigl\lvert f(\xn{j}_0) \bigr\rvert^2
    + CT \int_{0}^{t} \revise{\Bigl\lvert \cov(\mu^J_s) : \hessian f(\xn{j}_s) \Bigr\rvert^{2}} \, \d s \\
    &\qquad \qquad
    + 2C_{\rm BDG} \int_{0}^{t} \expect \Bigl\lvert \nabla f(\xn{j}_s)^\t \cov(\mu^J_s) \nabla f(\xn{j}_s) \Bigr\rvert \, \d s.
    \label{eq:bdg_ips}
\end{align}
\revise{%
    By~\eqref{eq:aux_hessian} and H\"older's inequality (recall that $p \geq 4$),
    it holds that
    \begin{align*}
        \expect \Bigl\lvert \cov(\mu^J_s) : \hessian f(\xn{j}_s) \Bigr\rvert^{2}
        &\leq
        C\expect \Bigl[ \bigl\lVert \cov(\mu^J_s) \bigr\rVert_{\rm F}^{2} \left\lvert \xn{j}_s \right\rvert_*^{2q(\ell + 2) - 4} \Bigr] \\
        &\leq
        C\expect \Bigl[ \bigl\lVert \cov(\mu^J_s) \bigr\rVert_{\rm F}^{q(\ell + 2)} \Bigr]^{\frac{4}{2q(\ell + 2)}} \, \expect \Bigl[ \left\lvert \xn{j}_s \right\rvert_*^{2q(\ell + 2)} \Bigr]^{\frac{2q(\ell + 2) - 4}{2q(\ell + 2)}}
        \leq
        C \expect \Bigl[ \left\lvert \xn{j}_s \right\rvert_*^{2q(\ell + 2)} \Bigr].
    \end{align*}
    In the last inequality, we used that
    $\cov(\mu^J_s) \preccurlyeq \frac{1}{J} \sum_{j=1}^{J} \xn{j}_s \otimes \xn{j}_s$ and $\lVert x \otimes x \rVert_{\rm F} = |x|^2$,
    so by Jensen's inequality and exchangeability it holds that
    \[
        \expect \Bigl[ \bigl\lVert \cov(\mu^J_s) \bigr\rVert_{\rm F}^{q(\ell + 2)} \Bigr]
        \leq \expect  \Biggl[ \biggl(\frac{1}{J} \sum_{j=1}^{J} \left\lvert \xn{j}_s \right\rvert^2 \biggr)^{q(\ell + 2)} \Biggr]
        \leq \expect  \Biggl[ \frac{1}{J} \sum_{j=1}^{J} \left\lvert \xn{j}_s \right\rvert^{2q(\ell + 2)} \Biggr]
        = \expect \left\lvert \xn{j}_s \right\rvert_*^{2q(\ell + 2)}.
    \]
Likewise,} since $\bigl\lvert \nabla f(x) \bigr\rvert \leq C  |x|_*^{q(\ell + 2) - 1}$ for all $x \in \real^d$ by~\eqref{assump:grad_bounds},
\revise{we have using the same reasoning that}
\begin{align}
    \nonumber
    \expect \Bigl\lvert \nabla f(\xn{j}_s)^\t \cov(\mu^J_s) \nabla f(\xn{j}_s) \Bigr\rvert
    &
    \revise{\leq \expect \left[ \bigl\lVert \cov(\mu^J_s) \bigr\rVert_{\rm F} \left\lvert \xn{j}_s \right\rvert_*^{2q(\ell + 2) - 2} \right] }
    \leq C\expect \lvert \xn{j}_s \rvert_*^{2q(\ell + 2)}.
\end{align}
Substituting this bound in~\eqref{eq:bdg_ips} and using~\eqref{assump:function_bounds},
we deduce that
\[
    \expect \left[ \sup_{s \in [0, t]} \bigl\lvert \xn{j}_s \bigr\rvert_*^{2q(\ell + 2)} \right]
    \leq \expect \left[ \bigl\lvert \xn{j}_0 \bigr\rvert_*^{2q(\ell + 2)} \right]
    + C \int_{0}^{t} \expect \bigl\lvert \xn{j}_s \bigr\rvert_*^{2q(\ell + 2)} \, \d s.
\]
The moment bound~\eqref{eq:moment_bounds_2} then follows from~\eqref{eq:khasminskii_bound},
or from an application of Gr\"onwall's inequality.

\subsection{Auxiliary lemmas to establish well-posedness of the mean field dynamics}
\label{sub:aux_wellposed}

In the following, we endow the vector space $\mathcal X := \R^{d \times d}$ with the Frobenius norm,
and we let $\mathcal{R}\colon \R^{d\times d} \rightarrow \R^{d\times d}$ be the map defined by $\mathcal R(\Gamma) = \sqrt{\Gamma \Gamma^\t}$.
We prove auxiliary lemmas in this section,
and postpone the proof of~\cref{proposition:well-posedness} to~\cref{sub:proof_wellposed}.

\begin{lemma}
    \label{lemma:cont-moments}
    Suppose that~$\phi \in \mathcal A(\ell)$ and that $\mfl_0 \in \mathcal P_{p}(\real^d)$ for $p \geq \ell + 2$.
    Fix $\Gamma \in \cont([0,T], \mathcal X)$ and $y_0 \sim \mfl_0$.
    Then there is a unique strong solution~$Y\in \cont([0,T], \R^d)$ to the stochastic differential equation
    \begin{align}
        \label{eq:aux_lemma_SDE}
        \d Y_t = - \mathcal R (\Gamma_t) \nabla \phi(Y_t) \, \d t + \sqrt{2\mathcal R (\Gamma_t)} \, \d W_t,
        \qquad Y_0 = y_0\,,
    \end{align}
    In addition, the function $\mathcal K \colon [0, T] \to \mathcal X$ given by $\mathcal K(t) = \cov(\rho_t)$,
    where $\rho_t = {\rm Law}(Y_t)$,
    belongs to $\cont^1([0, T], \mathcal X)$,
    and there is~$\sigma > 0$ depending on  $\lVert \Gamma \rVert_{\cont\left([0, T], \mathcal X\right)}$ such that
    \begin{equation}
        \label{eq:bound_derivative_coviariance}
        \lVert \mathcal K \rVert_{\cont^1\bigl([0, T], \mathcal X\bigr)}
        \leq
        \sigma.
    \end{equation}
\end{lemma}
\begin{proof}
    The existence of a unique continuous strong solution~$Y\in \cont([0,T], \R^d)$ to~\eqref{eq:aux_lemma_SDE} follows from classical theory of stochastic differential equations,
    for example from~\cite[Theorem 3.5]{khasminskii2013stochastic} with the Lyapunov function~$x \mapsto \phi(x)^{\frac{2}{\ell + 2}}$.
    It remains to prove~\eqref{eq:bound_derivative_coviariance}.
    By It\^o's formula,
    it holds with $f(x) = \phi(x)^q$ for any $q > 0$ that
    \begin{align}
        \d f(Y_t)
          \label{eq:ito_power_y}
          &= - q  \phi(x) ^{q-1} \nabla \phi(Y_t)^\t \mathcal R(\Gamma_t) \nabla \phi(Y_t) \, \d t + \mathcal R(\Gamma_t) : \hessian f(Y_t) \, \d t + \nabla f(Y_t)^\t \sqrt{2 \mathcal R(\Gamma_t)} \, \d W_t.
    \end{align}
    By~\eqref{assump:grad_bounds} in~\cref{assumption:main},
    it holds for all $x \in \real^d$ that $\bigl\lvert \nabla f(x) \bigr\rvert \leq C |x|_*^{q(\ell + 2) - 1}$.
    Thus, writing~\eqref{eq:ito_power_y} in integral form,
    then taking the supremum and using the Burkholder--Davis--Gundy inequality,
    we obtain for all $t \in [0, T]$ that
    \begin{align*}
        \frac{1}{3} \expect \left[ \sup_{s \in [0, t]} \bigl\lvert f(Y_s) \bigr\rvert^2 \right]
        &\leq \expect \bigl\lvert f(Y_0)\rvert^2 + T \int_{0}^{t} \Bigl\lVert \mathcal R(\Gamma_s) \Bigr\rVert_{\rm F}^2 \expect \Bigl\lVert \hessian f(Y_s) \Bigr\rVert_{\rm F}^2  \, \d s \\
        &\qquad + 2 C_{\rm BDG} \int_{0}^{t}  \expect \left\lvert  \nabla f(Y_s)^\t \mathcal R\bigl(\Gamma(s)\bigr) \nabla f(Y_s) \right\rvert \, \d s
        \\
        &\leq C \expect \lvert Y_0 \rvert_*^{2q(\ell + 2)} + C\left(\norm{\Gamma}^2_{\cont([0, T], \mathcal X)} + 1 \right) \int_{0}^{t} \expect \lvert Y_s \rvert_*^{2q(\ell + 2) - 4} + \expect \lvert Y_s \rvert_*^{2q(\ell + 2) - 2} \, \d s.
    \end{align*}
    Using the lower bound in~\eqref{assump:function_bounds} and rearranging,
    we finally obtain
    \[
        \expect \left[ \sup_{s \in [0, t]} \bigl\lvert Y_{\revise{s}} \bigr\rvert_*^{2q(\ell + 2)} \right]
        \leq C \expect \lvert Y_0 \rvert_*^{2q(\ell + 2)} + C\left(\norm{\Gamma}^2_{\cont([0, T], \mathcal X)} + 1 \right) \int_{0}^{t} \expect \left[ \sup_{u \in [0, s]} \lvert Y_u \rvert_*^{2q(\ell + 2)} \right] \, \d s.
    \]
    From Gr\"onwall's inequality,
    it follows that
    \begin{equation}
        \label{eq:first_gronwall_well_posed}
        \expect \left[ \sup_{s \in [0, \revise{T}]} \bigl\lvert Y_{\revise{s}} \bigr\rvert_*^{2q(\ell + 2)} \right]
        \leq \revise{C \expect \left\lvert Y_0 \right\rvert_*^{2q(\ell + 2)}}
    \end{equation}
    for some \revise{constant $C$ depending on $\norm{\Gamma}_{\cont([0, T], \mathcal X)}$ and~$T$,
    among other parameters}.
    In particular, since $\mfl_0 \in \mathcal P_{\ell + 2}(\real^d)$,
    this inequality with~$q = \frac{1}{\ell + 2}$ together with dominated convergence imply that
    the functions $t \mapsto \expect [Y_t]$ and $t \mapsto \expect [Y_t \otimes Y_t]$ are continuous.
    Furthermore, using~\eqref{eq:first_gronwall_well_posed} with $q = \frac{1}{2}$
    we deduce that these functions are also differentiable on~$[0, T]$ because,
    by It\^o's lemma and dominated convergence,
    \begin{align}
        \label{eq:derivative_mean}
        \frac{1}{h} \Bigl( \expect [Y_{t+h}] - \expect [Y_t] \Bigr) &= - \frac{1}{h} \expect \left[ \int_{t}^{t+h} \mathcal R (\Gamma_u) \nabla \phi(Y_u) \, \d u \right]
        \xrightarrow[h \to 0]{} - \expect \Bigl[ \mathcal R (\Gamma_t) \nabla \phi(Y_t) \Bigr].
    \end{align}
    \revise{Similarly, applying It\^o's formula in the same manner as in~\eqref{eq:ito_cov} below},
    we obtain
    \begin{equation}
        \label{eq:derivative_square}
        \lim_{h \to 0} \frac{1}{h} \Bigl( \expect [Y_{t+h} \otimes Y_{t+h}] - \expect [Y_t \otimes Y_t] \Bigr)
        =
        - \expect \Bigl[ \bigl(\mathcal R (\Gamma_t) \nabla \phi(Y_t)\bigr) \otimes Y_t +  Y_t \otimes \bigl(\mathcal R (\Gamma_t) \nabla \phi(Y_t)\bigr) \Bigr]
        + 2\mathcal R (\Gamma_t).
    \end{equation}
    Another application of dominated convergence yields continuity of the derivatives.
    Finally,
    we bound the right-hand side of~\eqref{eq:derivative_mean} as follows:
    \begin{align*}
        \Bigl\lvert \expect \left[ \mathcal R (\Gamma_t) \nabla \phi(Y_t) \right] \Bigr\rvert
        &\leq \expect \Bigl[ \bigl\lVert \Gamma_t \bigr\rVert_{\rm F} \bigl\lvert \nabla \phi(Y_t) \bigr\rvert \Bigr]
        \leq  \u \bigl\lVert\Gamma \bigr\rVert_{\cont([0, T], \mathcal X)}  \expect  \bigl\lvert Y_t \bigr\rvert_*^{\ell + 1}  .
    \end{align*}
    Employing a similar reasoning for the right-hand side of~\eqref{eq:derivative_square},
    and using~\eqref{eq:first_gronwall_well_posed},
    we finally obtain~\eqref{eq:bound_derivative_coviariance}.
\end{proof}

\begin{lemma}
    \label{lemma:second_auxiliary_wellposedness}
    Suppose that $\phi \in \mathcal A(\ell)$ for $\ell \geq 0$ and $\mfl_0 \in \mathcal P_p(\real^d)$ for some \revise{$p \geq 4$}.
    Fix $x_0 \sim \mfl_0$ and $\xi \in [0, 1]$,
    and suppose that $\xl \in C\left([0,T], \R^d\right)$
    is a strong solution to
    \begin{equation}
        \label{eq:mckean_withxi}
        \left\{
        \begin{aligned}
            \d\xl_t &= - \xi \cov (\mfl_t) \nabla \phi(\xl_t) \, \d t + \sqrt{2 \xi \cov (\mfl_t) }  \, \d W_t,\\
            \mfl_t &= \Law(\xl_t),
        \end{aligned}
        \right.
    \end{equation}
    with initial condition $\xl_0 = x_0$.
    Then there is $\kappa > 0$ independent of~$\xi$ such that
    \begin{align}
        \label{eq:moment_bounds_mf}
        &\expect \left[\sup_{t \in [0, T]} \bigl\lvert \xnl{j}_t \bigr\rvert^{p} \right] \leq \kappa,
    \end{align}
    Finally, if \revise{furthermore} $p \geq \ell + 2$ and  $\cov (\mfl_0) \succ 0$,
    then there is $\eta > 0$ such that
    \begin{equation}
        \label{eq:no_collapse}
        \forall t \in [0, T],
        \qquad
        \cov(\mfl_t) \succcurlyeq \eta \I_d.
    \end{equation}
\end{lemma}

\begin{proof}
    The statements~\eqref{eq:moment_bounds_mf} and~\eqref{eq:no_collapse} are obvious if $\xi = 0$.
    The proof for a fixed value of $\xi \in (0, 1]$ is the same for all values of $\xi$,
    \revise{with constants that can easily be fixed independently of $\xi$,}
    so for simplicity we assume from now on that $\xi = 1$.

    \paragraph{Proof of the bound~\eqref{eq:moment_bounds_mf}}
    The proof of this bound is similar to the proof of~\eqref{eq:first_gronwall_well_posed} in~\cref{lemma:cont-moments}.
    Let $f(x) = \phi(x)^q$ \revise{for $q = \frac{p}{2(\ell + 2)}$}.
    By It\^o's formula, it holds that
    \begin{align}
        \notag
        \d f(\xl_t)
          &= - \nabla \phi(\xl_t)^\t \cov(\mfl_t) \nabla f(\xl_t) \, \d t + \cov(\mfl_t) : \hessian f(\xl_t) \, \d t + \nabla f(\xl_t)^\t \sqrt{2 \cov(\mfl_t)} \, \d W_t \\
          \label{eq:ito_power}
          &= - q  \phi(\xl_t) ^{q-1} \nabla \phi(\xl_t)^\t \cov(\mfl_t) \nabla \phi(\xl_t) \, \d t + \cov(\mfl_t) : \hessian f(\xl_t) \, \d t + \nabla f(\xl_t)^\t \sqrt{2 \cov(\mfl_t)} \, \d W_t.
    \end{align}
    Writing this equation in integral form,
    taking the supremum and using the Burkholder--Davis--Gundy inequality,
    we obtain that
    \begin{align*}
        \frac{1}{3} \expect \left[ \sup_{s \in [0, t]} \bigl\lvert f(\xl_t) \bigr\rvert^2 \right]
        &\leq \expect  \bigl\lvert f(\xl_0) \bigr\rvert^2
        + T \int_{0}^{t} \bigl\lVert \cov(\mfl_s) \bigr\rVert_{\rm F}^2 \expect \bigl\lVert \hessian f(\xl_s) \bigr\rVert_{\rm F}^2 \, \d s \\
        &\qquad + 2C_{\rm BDG} \int_{0}^{t}  \expect \left\lvert  \nabla f(\xl_s)^\t \cov(\mfl_s) \nabla f(\xl_s) \right\rvert \, \d s.
    \end{align*}
    Let us bound the terms in the integrals on the right-hand side.
    Using the inequality $\bigl\lVert \cov(\mfl_t) \bigr\rVert_{\rm F} \leq \expect \lvert \xl_t \rvert^2$, the bound~\eqref{eq:aux_hessian},
    and then Jensen's inequality,
    we obtain
    \begin{align*}
        \bigl\lVert \cov(\mfl_s) \bigr\rVert_{\rm F}^2 \expect \bigl\lVert \hessian {f(\xl_s)} \bigr\rVert_{\rm F}^2
        &\leq
        C
        \left(\expect  \bigl\lvert \xl_s \bigr\rvert^2  \right)^2
        \expect \bigl\lvert \xl_s \bigr\rvert_*^{2q(\ell + 2) - 4} \\
        &\leq C \left(\expect  \bigl\lvert \xl_s \bigr\rvert^{2q(\ell + 2)}  \right)^{\frac{4}{2q(\ell + 2)}}
        \left(\expect  \bigl\lvert \xl_s \bigr\rvert_*^{2q(\ell + 2)}  \right)^{\frac{2q(\ell + 2) - 4}{2q(\ell + 2)}}
        = C\expect  \bigl\lvert \xl_s \bigr\rvert_*^{2q(\ell + 2)}.
    \end{align*}
    By a similar reasoning,
    it holds that
    \[
        \expect \left\lvert  \nabla f(\xl_s)^\t \cov(\mfl_s) \nabla f(\xl_s) \right\rvert
        \leq C\expect  \bigl\lvert \xl_s \bigr\rvert_*^{2q(\ell + 2)}.
    \]
    Therefore, using the lower bound in~\eqref{assump:function_bounds},
    we deduce that
    \[
        \expect \left[ \sup_{s \in [0, t]} \left\lvert \xl_s \right\rvert_*^{2q(\ell + 2)} \right]
        \leq C \expect \left\lvert \xl_0 \right\rvert_*^{2q(\ell + 2)}
        + C \int_{0}^{t} \expect  \left[ \sup_{u \in [0, s]} \bigl\lvert \xl_u \bigr\rvert_*^{2q(\ell + 2)} \right] \d s.
    \]
    \revise{Using} Gr\"onwall's inequality,
    we obtain~\eqref{eq:moment_bounds_mf}.

\paragraph{Proof of~\eqref{eq:no_collapse}}
Let $g(x) = x \otimes x$. By It\^o's formula, it holds that
\begin{align}
    \notag
    \d g(\xl_t)
    &=
    - \bigl(\cov(\mfl_t) \nabla \phi(\xl_t)\bigr) \otimes \xl_t \, \d t
    - \xl_t \otimes \bigl(\cov(\mfl_t) \nabla \phi(\xl_t)\bigr) \, \d t
    + 2 \cov(\mfl_t) \, \d t \\
    \label{eq:ito_cov}
    &\qquad
    + \xl_t \otimes \bigl(\sqrt{2\cov(\mfl_t)} \, \d W_t\bigr)
    + \bigl(\sqrt{2\cov(\mfl_t)} \, \d W_t\bigr) \otimes \xl_t \, .
\end{align}
Therefore, taking expectations, we deduce that
\[
    \expect g(\xl_t) - \expect g(\xl_0)
    = \int_{0}^{t} \revise{2}\cov(\mfl_s) - \expect \Bigl[ \bigl(\cov(\mfl_s) \nabla \phi(\xl_s)\bigr) \otimes \xl_s
    + \xl_s \otimes \bigl(\cov(\mfl_s) \nabla \phi(\xl_s)\bigr) \Bigr] \, \d s \,.
\]
On the other hand
\[
    \expect [\xl_t] - \expect [\xl_0]
    = - \int_{0}^{t}  \cov(\mfl_s) \expect \left[ \nabla \phi(\xl_s) \right]  \, \d s \, .
\]
Combining these equations,
and noting that $t \mapsto \cov(\mfl_t)$ is differentiable by~\cref{lemma:cont-moments},
we deduce that
\begin{align*}
    \frac{\d}{\d t}\cov(\mfl_t)
    &=
    2\cov(\mfl_t)
    - \expect \left[ \Bigl(\cov(\mfl_t) \nabla \phi(\xl_t)\Bigr) \otimes \Bigl( \xl_t - \mathcal M(\mfl_t) \Bigr) \right]
    - \expect \left[ \Bigl( \xl_t - \mathcal M(\mfl_t) \Bigr) \otimes \Bigl(\cov(\mfl_t) \nabla \phi(\xl_t)\Bigr)  \right].
\end{align*}
Inspired by~\cite[Lemma A.1]{MR4123680},
we define the Lyapunov function
\[
    \mathcal V_{\cov}(\mu) =  - \log \det \bigl(\cov(\mu) \bigr).
\]
Using Jacobi's formula for the determinant,
we deduce that
\[
    \frac{\d}{\d t} \mathcal V_{\cov}(\mfl_t) =
    \trace \left( \cov(\mfl_t)^{-1} \frac{\d}{\d t} \cov(\mfl_t)\right)
    = \trace \Bigl( \revise{2} \I_d - 2 \expect \Bigl[ \nabla \phi(\xl_t) \otimes \bigl(\xl_t - \mean(\mfl_t) \bigr) \Bigr] \Bigr).
\]
It follows that
\begin{align*}
    \mathcal V_{\cov}(\mfl_t) - \mathcal V_{\cov}(\mfl_0)
    &\leq \revise{2} t d  + 2 \sqrt{d} \int_{0}^{t} \expect \Bigl[ \left\lVert \nabla \phi(\xl_s) \otimes \bigl(\xl_s - \mean(\mfl_s) \bigr) \right\rVert_{\rm F} \Bigr] \\
    &= \revise{2} t d + 2 \sqrt{d} \int_{0}^{t} \expect \Bigl[ \left\lvert \nabla \phi(\xl_s) \right\rvert \cdot \left\lvert \xl_s - \mean(\mfl_s) \right\rvert \Bigr] \, \d s,
\end{align*}
\revise{where we used that $\trace A \leq \sqrt{d} \norm{A}_{\rm F}$ for any matrix~$A$.}
The integrand can be bounded from~\eqref{eq:moment_bounds_mf} with \revise{$p = 4 \vee (\ell + 2)$},
which concludes the proof since for any symmetric positive definite matrix~$A$
it holds that
\[
    \det(A) \leq \lambda_{\min}(A) \lambda_{\max}(A)^{d-1}
    \qquad \Rightarrow \qquad
    \lambda_{\min}(A) \geq \frac{\det A}{\lVert A \rVert_{\rm F}^{d-1}},
\]
where $\lambda_{\min}(A)$ and $\lambda_{\max}(A)$ are the minimum and maximum eigenvalues of~$A$.
\end{proof}

\subsection{\texorpdfstring{Proof of~\cref{proposition:well-posedness}}{Proof of well-posedness theorem}}
\label{sub:proof_wellposed}

The proof of well-posedness of the mean field dynamics is based on a classical fixed point argument,
applied in the vector space $\cont\bigl([0,T], \real^{d \times d}\bigr)$.
Similarly to the proof of \cite[Theorem 3.1]{carrillo2018analytical},
we first construct a map
\[
    \mathcal{T}\colon \cont\left([0,T], \mathcal X \right) \rightarrow  \cont\left([0,T], \mathcal X \right)
\]
whose fixed points correspond to solutions of \eqref{eq:mckean}.
This section follows closely the proof of~\cite[Theorem 2.4]{gerber2023meanfield}.

\paragraph{Step 1. Constructing the map~$\mathcal T$}
Consider the map
\begin{align*}
    \mathcal{T}\colon \cont([0,T], \mathcal X) &\rightarrow  \cont([0,T], \mathcal X) \\
    \Gamma &\mapsto \Bigl( t \mapsto \cov (\rho_t) \Bigr),
\end{align*}
where $\rho_t := {\rm Law}(Y_t)$ and $(Y_t)_{t \in [0, T]}$ is the unique solution to~\eqref{eq:aux_lemma_SDE} with matrix $\Gamma$.
By~\cref{lemma:cont-moments}, the map $\mathcal{T}$ is well-defined.
Fixed points of $\mathcal{T}$ correspond to solutions of the McKean-Vlasov SDE~\eqref{eq:mckean}.
The existence of a fixed point follows from applying the Leray-Schauder fixed point theorem \protect{\cite[Chapter 11]{gilbarg1977elliptic}} in the space~$C\left([0,T], \mathcal X \right)$,
once we have proved that $\mathcal T$ is compact and that the following set is bounded:
\begin{align}
    \label{eq:leray-schauder-set}
    \Bigl\{\Gamma \in \cont\bigl([0,T], \mathcal X\bigr) :  \exists \xi\in[0,1]\text{ such that } \Gamma=\xi \mathcal{T}\left(\Gamma\right)\Bigr\}.
\end{align}

\paragraph{Step 2. Showing that~$\mathcal T$ is compact}
To prove that $\mathcal{T}$ is a compact operator,
fix $R>0$ and consider the ball
\[
    B_R := \left\{ \Gamma\in  \cont([0,T], \mathcal X) : \norm{\Gamma}_{\cont([0, T], \mathcal X)} \le  R \right\}.
\]
By the Arzelà–Ascoli theorem,
\revise{the following embedding is compact}:
\begin{align*}
    \cont^{1}\left([0,T], \mathcal X\right) \hookrightarrow \cont\left([0,T], \mathcal X\right).
\end{align*}
Thus, it suffices to show that $\mathcal{T}(B_R)$ is bounded in $\cont^{1}\left([0,T], \mathcal X\right)$,
which follows immediately from the assertion~\eqref{eq:bound_derivative_coviariance} in~\cref{lemma:cont-moments}.

\paragraph{Step 3. Showing that the set~\eqref{eq:leray-schauder-set} is bounded}
To this end,
assume that $\Gamma \in  \cont([0,T], \mathcal X)$ satisfies
\begin{equation}
    \label{eq:assumption_boundedness}
    \Gamma = \xi \mathcal{T}(\Gamma)
\end{equation}
for some $\xi\in[0,1]$,
and let $(Y_t)$ denote the corresponding solution to~\eqref{eq:aux_lemma_SDE}.
By \eqref{eq:assumption_boundedness},
the stochastic process~$(Y_t)$ is also a solution to
\[
    \d Y_t = - \xi \cov(\rho_t) \nabla \phi(Y_t) \, \d t + \sqrt{2 \xi \cov(\rho_t)} \, \d W_t,
    \qquad \rho_t = {\rm Law}(Y_t)\,.
\]
By \cref{lemma:second_auxiliary_wellposedness},
it holds that $\norm{\cov(\rho_t)}$ is bounded uniformly in $[0, T]$ by a constant independent of~$\xi$,
and so the set~\eqref{eq:leray-schauder-set} is indeed bounded.
This establishes the existence of a fixed point of~$\mathcal T$.
Furthermore, \cref{lemma:second_auxiliary_wellposedness} yields the first and second moment bounds in~\eqref{eq:moment_bound_mfl},
whereas~\eqref{eq:bound_derivative_coviariance} in~\cref{lemma:cont-moments} yields the third bound in~\eqref{eq:moment_bound_mfl}.

\paragraph{Step 4a. Showing uniqueness when $\phi \in \mathcal A(0)$}
\revise{We first present a simple proof of uniqueness for the case~$\phi \in \mathcal A(0)$,
and postpone the general proof to \textbf{Step 4b}.}

Let $\Gamma$ and $\widehat{\Gamma}$ be two fixed points of $\mathcal{T}$ with corresponding solutions~$Y_t, \widehat{Y}_t$ of \eqref{eq:aux_lemma_SDE}.
By definition of $\mathcal{T}$ and since~$\cov(\mu)$ is symmetric and positive semidefinite for all~$\mu\in\mathcal{P}_2(\R^d)$,
it holds that $\mathcal{R}\left(\Gamma_t\right) = \Gamma_t$ and $\mathcal{R}\bigl(\widehat{\Gamma}_t\bigr) = \widehat{\Gamma}_t$ for all~$t\in[0,T]$.
Let~$\rho_t, \widehat \rho_t \in \mathcal P(\R^d)$ denote the marginal laws of $Y_t$ and $\widehat Y_t$,
respectively.
By the Burkholder--Davis--Gundy inequality,
we have for all $t \in [0, T]$
\begin{align}
    \frac{1}{2}\expect \left[
    \sup_{s \in [0, t]} \bigl\lvert Y_{s} - \widehat{Y}_{s} \bigr\rvert^{2} \right]
    &\leq T \int_{0}^{t} \expect  \left\lvert \cov (\rho_{s}) \nabla \phi(Y_{s}) - \cov (\widehat \rho_{s}) \nabla \phi(\widehat Y_{s}) \right\rvert^2 \, \d s
    \label{eq:bdg_well_posed_uniqueness}
    + 2 C_{\rm BDG} \int_{0}^t \left\lVert \sqrt{\cov(\rho_{s})} - \sqrt{\cov(\widehat \rho_{s})} \right\rVert_{\rm F}^2 \, \d s.
\end{align}
By \cref{lemma:wasserstein_stability_estimates},
together with the inequality $W_p(\rho_s, \widehat \rho_s)^p \leq \expect \lvert Y_s - \widehat Y_s \rvert^p$,
which follows from the definition of the Wasserstein distance,
we have that
\[
    \left\lVert \sqrt{\cov(\rho_{s})} - \sqrt{\cov(\widehat \rho_{s})} \right\rVert_{\rm F}^2
    \leq 2 W_2\bigl(\rho_{s}, \widehat \rho_{s}\bigr)^2
    \leq 2 \expect \lvert Y_s - \widehat Y_s \rvert^2.
\]
For the first term on the right-hand side of~\eqref{eq:bdg_well_posed_uniqueness},
we use the triangle inequality to obtain
\begin{align*}
    &\frac{1}{2} \expect \left\lvert \cov (\rho_{s}) \nabla \phi(Y_{s}) - \cov (\widehat \rho_{s}) \nabla \phi(\widehat Y_{s}) \right\rvert^2 \\
    &\qquad \leq
    \expect \left\lvert \bigl(\cov (\rho_{s}) - \cov (\widehat \rho_{s})\bigr) \nabla \phi(Y_{s})  \right\rvert^2
    + \expect \left\lvert \cov (\widehat \rho_{s}) \bigl( \nabla \phi(Y_{s}) - \nabla \phi(\widehat Y_{s}) \bigr)  \right\rvert^2 \\
    &\qquad \leq
    C \bigl\lVert \cov (\rho_{s}) - \cov (\widehat \rho_{s}) \bigr\rVert_{\rm F}^2
    \expect \lvert Y_s \rvert_*^2
    + \lipphi \bigl\lVert \cov (\widehat \rho_{s}) \bigr\rVert_{\rm F}^2 \, \expect \left\lvert Y_{s} - \widehat Y_{s} \right\rvert^2 \\
    &\qquad \leq C \Bigl( W_2(\rho_{s}, \delta_0) + W_2(\widehat \rho_{s}, \delta_0) \Bigr)^2
    W_2(\rho_{s}, \widehat \rho_{s})^2  +
    C \expect \left\lvert Y_{s} - \widehat Y_{s} \right\rvert^2
    \leq C \expect \left\lvert Y_{s} - \widehat Y_{s} \right\rvert^2,
\end{align*}
where we used the Lipschitz continuity~\eqref{eq:assump-f:lip-growth-gradient},
the Wasserstein stability estimate for the covariance in~\cref{lemma:wasserstein_stability_estimates},
and the moment bound established in~\cref{lemma:second_auxiliary_wellposedness}.
We conclude that for all $t \in [0, T]$,
it holds that
\[
    \expect \left[ \sup_{s \in [0, t]} \bigl\lvert Y_{s} - \widehat Y_{s} \bigr\rvert^{2}  \right]
    \leq C \int_{0}^{t} \expect \left[ \sup_{u \in [0, s]} \bigl\lvert Y_{u} - \widehat Y_{u} \bigr\rvert^{2} \right] \, \d s.
\]
By Gr\"onwall's lemma,
it follows that the left-hand side is 0 for all~$t \in [0, T]$,
and so $\Gamma_t = \widehat \Gamma_t$ for all~$t \in [0, T]$,
which concludes the proof of uniqueness.

\paragraph{Step 4b. Showing uniqueness when $\phi \in \mathcal A(\ell)$ with $\ell > 0$}
Recall that $Y_t$ and $\widehat Y_t$ satisfy
\begin{align*}
    \d Y_t &= - \Gamma_t \nabla \phi(Y_t) + \sqrt{2 \Gamma_t} \, \d W_t, \\
    \d \widehat Y_t &= - \widehat \Gamma_t \nabla \phi(\widehat Y_t) + \sqrt{2 \widehat \Gamma_t} \, \d W_t.
\end{align*}
Applying It\^o's formula to $f(y, \widehat y, t) = \frac{1}{2} \bigl( y - \widehat y \bigr)^\t \Gamma_t^{-1} \bigl( y - \widehat y \bigr)$,
and noting that $f(Y_0, \widehat Y_0, 0) = 0$,
we obtain that
\begin{align*}
    \expect f(Y_t, \widehat Y_t, t)
    &= - \expect \int_{0}^{t} \Bigl\langle Y_s - \widehat Y_s, \nabla \phi(Y_s) - \nabla \phi (\widehat Y_s) \Bigr\rangle \, \d s
    + \expect \int_{0}^{t} \Bigl\langle Y_s - \widehat Y_s, \Gamma_s^{-1} (\widehat \Gamma_s - \Gamma_s) \nabla \phi (\widehat Y_s) \Bigr\rangle \, \d s \\
    &\qquad
    + \int_{0}^{t}
    \begin{pmatrix}
        \Gamma_s & \sqrt{\Gamma_s} \sqrt{\widehat \Gamma_s} \\
        \sqrt{\widehat \Gamma_s} \sqrt{\Gamma_s} & \widehat \Gamma_s
    \end{pmatrix}
    :
    \begin{pmatrix}
        \Gamma_s^{-1} & - \Gamma_s^{-1} \\
        - \Gamma_s^{-1} & \Gamma_s^{-1}
    \end{pmatrix}
    \, \d s \\
    &\qquad \revise{-} \expect \int_{0}^{t} \frac{1}{2} \bigl( Y_s - \widehat Y_s \bigr)^\t \Gamma_s^{-1} \frac{\d \Gamma_s}{\d s} \Gamma_s^{-1} \bigl( Y_s - \widehat Y_s \bigr)^\t \, \d s.
\end{align*}
Let us bound the four terms on the right-hand side.
\begin{itemize}
    \item
        By~\cref{lemma:convexity},
        it holds that
        \[
            - \Bigl\langle Y_s - \widehat Y_s, \nabla \phi(Y_s) - \nabla \phi (\widehat Y_s) \Bigr\rangle
            \leq c_2  \left\lvert Y_s - \widehat Y_s \right\rvert^2.
        \]

    \item
        For the second term,
        we have from the Cauchy--Schwarz inequality
        \begin{align*}
            \expect \Bigl\langle Y_s - \widehat Y_s, \Gamma_s^{-1} (\widehat \Gamma_s - \Gamma_s) \nabla \phi (\widehat Y_s) \Bigr\rangle
            &\leq  \Bigl\lVert \Gamma_s^{-1} \Bigr\rVert_{\rm F} \Bigl\lVert\widehat \Gamma_s - \Gamma_s\Bigr\rVert_{\rm F}
            \left( \expect \lvert Y_s - \widehat Y_s \rvert^2 \right)^{\frac{1}{2}}
            \Bigl( \expect \lvert \nabla \phi(Y_s) \rvert^2 \Bigr)^{\frac{1}{2}} \\
            &\leq C \Bigl\lVert \cov(\widehat \rho_s) - \cov(\rho_s) \Bigr\rVert_{\rm F} \left( \expect \lvert Y_s - \widehat Y_s \rvert^2 \right)^{\frac{1}{2}}
            \leq C \expect \lvert Y_s - \widehat Y_s \rvert^2,
        \end{align*}
        where we used the moment bounds from~\cref{lemma:second_auxiliary_wellposedness},
        the Wasserstein stability estimate~\eqref{eq:stab_wcov_simple} from~\cref{lemma:wasserstein_stability_estimates},
        and the inequality $W_2(\widehat \rho_s, \rho_s)^2 \leq \expect \bigl\lvert \widehat Y_s - Y_s \bigr\rvert^2$,
        which follows from the definition of the Wasserstein distance.
        \revise{%
            Note that $\expect \lvert \nabla \phi(Y_s) \rvert^2$ is indeed finite by~\eqref{eq:moment_bounds_mf}
            and the assumption that $\mfl_0 \in \mathcal P_{2(\ell + 2)}$.
        }

    \item
        For the third term,
        a simple calculation gives that
        \begin{align*}
            \begin{pmatrix}
                \Gamma_s & \sqrt{\Gamma_s} \sqrt{\widehat \Gamma_s} \\
                \sqrt{\widehat \Gamma_s} \sqrt{\Gamma_s} & \widehat \Gamma_s
            \end{pmatrix}
            :
            \begin{pmatrix}
                \Gamma_s^{-1} & - \Gamma_s^{-1} \\
                - \Gamma_s^{-1} & \Gamma_s^{-1}
            \end{pmatrix}
            &= \trace \left( \Bigl(\sqrt{\Gamma_s} - \sqrt{\widehat \Gamma_s}\Bigr) \Bigl(\sqrt{\Gamma_s} - \sqrt{\widehat \Gamma_s}\Bigr) \Gamma_s^{-1} \right) \\
            &= \left\lVert \sqrt{\Gamma_s^{-1}}\Bigl(\sqrt{\Gamma_s} - \sqrt{\widehat \Gamma_s}\Bigr) \right\rVert_{\rm F}^2
            \leq C \expect \left\lvert Y_t - \widehat Y_t \right\rvert^2.
        \end{align*}
        where we used the estimate~\eqref{eq:no_collapse} in~\cref{lemma:second_auxiliary_wellposedness} and
        the Wasserstein stability estimate~\eqref{eq:wmean-wcov-emp-local-lip} from \cref{lemma:wasserstein_stability_estimates} in the last inequality.

    \item
        Finally,
        it immediately follows from the bounds on $\Gamma_s = \cov(\rho_s)$ and $\widehat \Gamma_s = \cov(\widehat \rho_s)$ given~\cref{lemma:second_auxiliary_wellposedness},
        as well as the upper bound~\eqref{eq:bound_derivative_coviariance} on the time derivative of $\Gamma_s$,
        that
        \[
            \left\lvert \frac{1}{2} \bigl( Y_s - \widehat Y_s \bigr)^\t \Gamma_s^{-1} \frac{\d \Gamma_s}{\d s} \Gamma_s^{-1} \bigl( Y_s - \widehat Y_s \bigr)^\t \right\rvert
            \leq C \left\lvert Y_s - \widehat Y_s \right\rvert^2.
        \]
\end{itemize}
Combining all the bounds,
we deduce that
\[
    \expect \left\lvert Y_t - \widehat Y_t \right\rvert^2
    \leq C \expect f(Y_t, \widehat Y_t, t)
    \leq C \int_{0}^{t} \expect \left\lvert Y_s - \widehat Y_s \right\rvert^2 \, \d s.
\]
Gr\"onwall's inequality then gives that the left-hand side is 0 for all times,
which concludes the proof.

\section{Proof of auxiliary results}
\label{sec:proof_of_aux}

\subsection{\texorpdfstring{Proof of~\cref{lemma:wasserstein_stability_estimates}}{Proof of auxiliary lemma on Wasserstein stability estimate}}
\label{sub:proof_of_stab}

We prove the statements separately,
using exactly the same approach as in~\cite{gerber2023meanfield}.

\paragraph{Proof of~\eqref{eq:stab_wcov_simple}}
By the triangle inequality,
it holds that
\begin{equation}
    \label{eq:two_terms_stab_cov}
    \bigl\lVert \cov(\mu) - \cov(\nu) \bigr\rVert_{\rm F}
    \leq
    \bigl\lVert \mu [x\otimes x] - \nu[x \otimes x] \bigr\rVert_{\rm F}
    + \bigl\lVert \mean(\mu) \otimes \mean(\mu) - \mean(\nu) \otimes \mean(\nu) \bigr\rVert_{\rm F}.
\end{equation}
Let $\pi \in \Pi(\mu, \nu)$ denote an arbitrary coupling between $\mu$ and $\nu$.
By Jensen's inequality, it holds that
\[
    \bigl\lVert \mu [x\otimes x] - \nu[x \otimes x] \bigr\rVert_{\rm F}
    \leq \iint_{\R^d \times \R^d} \lVert x \otimes x - y \otimes y \rVert_{\rm F} \, \pi (\d x \, \d y).
\]
Recall that for all $(x, y) \in \R^d \times \R^d$, it holds that
\begin{align}
    \notag
    \left\lVert x \otimes x - y \otimes y \right\rVert_{\rm F}
        &= \left\lVert x (x - y)^\t + (x - y) y^\t \right\rVert_{\rm F} \\
        \label{eq:aux_weighted_mean}
        &\leq \left\lVert x (x - y)^\t \right\rVert_{\rm F} + \left\lVert (x - y) y^\t \right\rVert_{\rm F}
        = \abs{x} \cdot \abs{x - y} + \abs{y} \cdot \abs{x - y}
        = \bigl(\abs{x} + \abs{y}\bigr) \abs{x - y}.
\end{align}
Therefore, we deduce that
\[
    \bigl\lVert \mu [x\otimes x] - \nu[x \otimes x] \bigr\rVert_{\rm F}
    \leq
    \sqrt{\iint_{\R^d \times \R^d} \bigl(\abs{x} + \abs{y}\bigr)^2 \, \pi (\d x \, \d y)}
    \sqrt{\iint_{\R^d \times \R^d} \abs{x - y}^2 \, \pi (\d x \, \d y)}.
\]
Besides,
by the triangle inequality,
it holds that
\begin{equation*}
    \begin{aligned}
        \sqrt{\iint_{\R^d \times \R^d} \bigl(\abs{x} + \abs{y}\bigr)^2 \, \pi (\d x \, \d y)}
        \leq \sqrt{\int_{\R^d}\abs{x}^2 \, \mu(\d x)} + \sqrt{\int_{\R^d}\abs{y}^2 \, \nu(\d y)}
        = W_2(\mu,\delta_0) + W_2(\nu,\delta_0).
    \end{aligned}
\end{equation*}
Infimizing over all couplings $\pi \in \Pi(\mu, \nu)$,
we conclude that
\[
   \bigl\lVert \mu [x\otimes x] - \nu[x \otimes x] \bigr\rVert_{\rm F}
   \leq \Bigl( W_2(\mu, \delta_0) + W_2(\nu, \delta_0) \Bigr) W_2(\mu, \nu).
\]
Applying a similar reasoning to the second term on the right-hand side of~\eqref{eq:two_terms_stab_cov},
we deduce~\eqref{eq:stab_wcov_simple}.

\paragraph{Proof of~\eqref{eq:wmean-wcov-emp-local-lip}}
It is sufficient to check the claim for measures $\mu$ and $\nu$ of the form
\begin{equation}
    \label{eq:form_empirical}
    \mu^J = \frac{1}{J} \sum_{j=1}^{J} \delta_{\xn{j}},
    \qquad
    \nu^J = \frac{1}{J} \sum_{j=1}^{J} \delta_{\yn{j}},
    \qquad J \in \N^+.
\end{equation}
Indeed,
assume that the statement holds for all such pairs of probability measures,
and take $(\mu, \nu) \in \mathcal P_{2}(\real^d) \times \mathcal P_{2}(\real^d)$.
By~\cite[Theorem~6.18]{MR2459454},
there exists a sequence $\bigl\{(\mu^J, \nu^J)\bigr\}_{J \in \N^+}$ in $\mathcal P_{2}(\real^d) \times \mathcal P_{2}(\real^d)$
such that $W_2(\mu^J, \mu) \to 0$ and~$W_2(\nu^J, \nu) \to~0$ in the limit as $J \to \infty$.
Then
\begin{align*}
    \norm{
        \sqrt{\cov(\mu}) -
        \sqrt{\cov (\nu)}
    }_{\rm F}
    \leq
    &\norm{
        \sqrt{\cov(\mu^J)} -
        \sqrt{\cov(\nu^J)}
    }_{\rm F} \\
    &\qquad +
    \norm{
        \sqrt{\cov(\mu)} -
        \sqrt{\cov(\mu^J)}
    }_{\rm F}
    +
    \norm{
        \sqrt{\cov(\nu)} -
        \sqrt{\cov(\nu^J)}
    }_{\rm F}.
\end{align*}
The first term is bounded from above by $\sqrt{2}  W_2(\mu^J, \nu^J)$ by the base case,
while the other two terms converge to 0 in the limit as~$J \to \infty$ by~\eqref{eq:stab_wcov_simple} in~\cref{lemma:wasserstein_stability_estimates}.
Taking the limit $J \to \infty$, we deduce that
\[
    \norm{
        \sqrt{\cov(\mu)} -
        \sqrt{\cov(\nu)}
    }_{\rm F}
    \leq
    \sqrt{2}  W_2(\mu, \nu).
\]

\paragraph{Proof of the statement for empirical measures}
By~\cite[p.5]{MR1964483},
the Wasserstein distance between empirical measures $\mu^J$ and $\nu^J$ of the form~\eqref{eq:form_empirical}
is equal to
\[
    W_2(\mu^J, \nu^J)
    = \min_{\sigma \in \mathcal S_J} \left(\frac{1}{J} \sum_{j=1}^{J} \left\lvert \xn{j} - Y^{\sigma(j)} \right\rvert^2\right)^{\frac{1}{2}},
\]
where $\mathcal S_J$ denotes the set of permutations in $\{1, \dotsc, J\}$.
Thus, the claim will follow if we can prove that,
for any pair of probability measures $(\mu^J, \nu^J) \in \mathcal P_{2}(\real^d) \times \mathcal P_2(\real^d)$ of the form~\eqref{eq:form_empirical},
it holds that
\begin{equation}
    \label{eq:target_inequality}
    \norm{
        \sqrt{\cov(\mu^J)} -
        \sqrt{\cov (\nu^J)}
    }_{\rm F}
    \leq
    \sqrt{2} \, \left(\frac{1}{J} \sum_{j=1}^{J} \left\lvert \xn{j} - \yn{j} \right\rvert^p\right)^{\frac{1}{p}}.
\end{equation}
We henceforth drop the superscript~$J$ in $\mu^J, \nu^J$ for simplicity,
and write~$\mathbf X = \left(\xn{1}, \dotsc, \xn{J}\right)$ and~$\mathbf Y = \left(\yn{1}, \dotsc, \yn{J} \right)$.
The proof of~\eqref{eq:target_inequality} presented below follows the lines of a proof shown to me by N.\ J.\ Gerber,
who proved this inequality in preliminary work with F.\ Hoffmann which eventually lead to the preprint~\cite{gerber2023meanfield}.
First note that
\[
\cov(\mu)
= M_{\mathbf X} M_{\mathbf X}^\t, \qquad
M_{\mathbf X} := \frac{1}{\sqrt{J}}
\begin{pmatrix}
    \left(\xn{1} - \mean(\mu) \right)
    & \hdots &
    \left(\xn{J} - \mean(\mu) \right)
\end{pmatrix}.
\]
Proceeding in the same manner,
we construct a matrix $M_{\mathbf{Y}} \in \R^{d\times J}$ such that $\mean(\nu)=M_{\mathbf{Y}} M_{\mathbf{Y}}^\t$.
A result by Araki and Yamagami~\cite{Araki1981},
later generalized by Kittaneh~\cite{MR0787884} and Bhatia~\cite{MR1275617},
states for any two matrices~$A$ and $B$ with the same shape,
it holds that
\[
    \norm{\sqrt{A^\t A}-\sqrt{B^\t B}}_{\rm F} \le \sqrt{2}\norm{A-B}_{\rm F}.
\]
See also~\cite[Theorem VII.5.7]{MR1477662} for a textbook presentation.
This result,
applied with $A = M_{\bf X}$ and $B = M_{\bf Y}$,
yields
\begin{align*}
    \norm{
        \sqrt{\cov(\mu)} -
        \sqrt{\cov (\nu) }}_{\rm F}
    &\le
    \sqrt{2} \left(\frac{1}{J} \sum_{j=1}^J \abs[\Big]{
            \left(\xn{j} - \mean(\mu) \right)
    -  \left(\yn{j} - \mean(\nu) \right)}^2 \right)^{\frac{1}{2}} \\
    &= \min_{a \in \real^d}
    \sqrt{2} \left(\frac{1}{J} \sum_{j=1}^J \abs[\Big]{
            \xn{j} - \yn{j} - a}^2 \right)^{\frac{1}{2}}
    \le \sqrt{2}
    \left( \frac{1}{J} \sum_{j=1}^J \abs*{ \xn{j} -  \yn{j} }^2 \right)^{\frac{1}{2}},
\end{align*}
which shows~\eqref{eq:target_inequality} and completes the proof.

\subsection{\texorpdfstring{Proof of~\cref{lemma:convergence_covariance_iid}}{Proof of auxiliary lemma on Monte Carlo error}}
\label{sub:proof_of_convergence_covarianec_iid}
Equation~\eqref{eq:iid_convergence_sqrt} follows from~\eqref{eq:iid_convergence_cov},
and from an inequality due to van Hemmen and Ando~\cite{vanHemmen1980},
see also~\cite[Problem~X.5.5]{MR1477662},
which in view of the assumption $\cov(\mu) \succcurlyeq \eta \I_d \succ 0$
gives that
\[
    \left\lVert \sqrt{\cov(\empmfl)} - \sqrt{\cov(\mu)} \right\rVert_{\rm F}
    \leq \frac{1}{\eta} \Bigl\lVert \cov(\empmfl) - \cov(\mu) \Bigr\rVert_{\rm F}.
\]
The bound~\eqref{eq:iid_convergence_cov} follows from usual Monte Carlo estimates.
Using the triangle inequality, we have that
\begin{align*}
    \Bigl\lVert \cov(\empmfl) - \cov(\mu) \Bigr\rVert_{\rm F}^p
    &= \Bigl\lVert \empmfl \left[x \otimes x\right] - \mu \left[x \otimes x\right] - \mean(\empmfl) \otimes \mean(\empmfl) + \mean(\mu) \otimes \mean(\mu) \Bigr\rVert_{\rm F}^p \\
    &\leq 2^{p-1}\Bigl\lVert \empmfl \left[x \otimes x\right] - \mu \left[x \otimes x\right] \Bigr\rVert_{\rm F}^p
    + 2^{p-1}\Bigl\lVert \mean(\empmfl) \otimes \mean(\empmfl) - \mean(\mu) \otimes \mean(\mu) \Bigr\rVert_{\rm F}^p.
\end{align*}
Convergence to zero of the expectation of the first term with rate $J^{-\frac{p}{2}}$ follows from the Marcinkiewicz--Zygmund inequality.
For the second term,
we use~\eqref{eq:aux_weighted_mean}
to obtain that
\begin{align*}
    &\expect \Bigl\lVert \mean(\empmfl) \otimes \mean(\empmfl) - \mean(\mu) \otimes \mean(\mu) \Bigr\rVert_{\rm F}^p \\
    &\qquad \leq \expect \left[ \bigl(\abs*{\mean(\empmfl)} + \abs{\mean(\mu)}\bigr)^p \abs{\mean(\empmfl) - \mean(\mu)}^p \right] \\
    &\qquad \leq \left( \expect \left[ \bigl(\abs*{\mean(\empmfl)} + \abs{\mean(\mu)}\bigr)^{2p} \right] \expect \Bigl[ \abs[\big]{\mean(\empmfl) - \mean(\mu)}^{2p} \Bigr] \right)^{\frac{1}{2}} \\
    &\qquad \leq C J^{-\frac{p}{2}} \left( \expect \left[ \bigl(\abs*{\mean(\empmfl)} + \abs{\mean(\mu)}\bigr)^{2p} \right] \right)^{\frac{1}{2}},
\end{align*}
where we used again the Marcinkiewicz--Zygmund inequality.
The claim follows since $\expect \Bigl[ \left\lvert \mean(\empmfl) \right\rvert^{2p} \Bigr] \leq \expect \Bigl[ \bigl\lvert \xnl{j} \bigr\rvert^{2p} \Bigr]$ by Jensen's inequality.

\paragraph{Acknowledgements}
The author is grateful to Zhiyan Ding, Nicolai Gerber, Franca Hoffmann, Qin Li and Julien Reygner for useful discussions and suggestions,
\revise{and also to the anonymous reviewers for valuable suggestions}.
UV is partially supported by the European Research Council (ERC) under the European Union's Horizon 2020 research and innovation programme (grant agreement No 810367),
and by the Agence Nationale de la Recherche under grants ANR-21-CE40-0006 (SINEQ) and ANR-23-CE40-0027 (IPSO).

\printbibliography
\end{document}